\documentclass[11pt,usenames,dvipsnames]{extarticle}

\usepackage{amsfonts}
\usepackage{amsmath}
\usepackage{amssymb}
\usepackage{amscd}
\usepackage{amsthm}
\usepackage{mathrsfs}
\usepackage{graphicx}
\usepackage{wasysym}
\usepackage[shortlabels]{enumitem}
\usepackage{xcolor,colortbl}
\usepackage{tikz,tikz-cd}
\usepackage{geometry}
\usepackage{physics}
\usepackage[12hr]{datetime}
\usepackage{textcomp}
\usepackage[backref=page]{hyperref}
\usepackage{inputenc}
\usepackage{appendix}
\usepackage{mathtools}
\usepackage{xfrac}
\usepackage{nicefrac}
\usepackage{url}
\usepackage[linesnumbered,ruled,vlined]{algorithm2e}
\usepackage{titlesec}
\usepackage{nicematrix}

\hypersetup{
	colorlinks=true, urlcolor=NavyBlue, linkcolor=Mahogany, citecolor=ForestGreen, pdfborder={0,0,0},
}

\mathchardef\mhyphen="2D

\newcommand{\mb}{\mathbb}
\newcommand{\ms}{\mathscr}
\newcommand{\mc}{\mathcal}
\newcommand{\mf}{\mathfrak}

\newcommand{\tsf}{\textsf}

\newcommand{\msf}{\mathsf}
\newcommand{\n}{\enspace}
\newcommand{\tx}{\text}
\newcommand{\ol}{\overline}

\newcommand{\innext}{\msf{E}^{(\tx{\normalfont in})}}
\newcommand{\outext}{\msf{E}^{(\tx{\normalfont out})}}

\newcommand{\wt}{\widetilde}

\newcommand{\out}{\tx{\normalfont out}}
\newcommand{\inn}{\tx{\normalfont in}}
\newcommand{\innint}{\tx{\normalfont in-int}}
\newcommand{\outint}{\tx{\normalfont out-int}}
\newcommand{\coeff}{\tx{\normalfont coeff}}

\newcommand{\iref}[2]{(\hyperref[#2]{\ref*{#1}.\ref*{#2}})}

\newcommand{\mult}{\mathrm{mult}}
\newcommand{\EHC}{\msf{EHC}}
\newcommand{\bEHC}{\msf{b\mhyphen EHC}}
\newcommand{\EPC}{\msf{EPC}}
\newcommand{\bEPC}{\msf{b\mhyphen EPC}}

\newcommand{\email}[1]{\href{mailto:#1}{\textcolor{NavyBlue}{\texttt{#1}}}}

\newcommand{\sep}{\tx{\normalfont sep}}

\theoremstyle{theorem}
\newtheorem{theorem}{Theorem}[section]
\newtheorem{fact}[theorem]{Fact}

\newtheorem{proposition}[theorem]{Proposition}
\newtheorem{observation}[theorem]{Observation}
\newtheorem{lemma}[theorem]{Lemma}
\newtheorem{corollary}[theorem]{Corollary}
\newtheorem{conjecture}[theorem]{Conjecture}
\newtheorem{question}[theorem]{Question}

\theoremstyle{definition}
\newtheorem{remark}[theorem]{Remark}
\newtheorem{example}[theorem]{Example}

\newtheoremstyle{TheoremNum}
{\topsep}{\topsep}              %%% space between body and thm
{\itshape}                      %%% Thm body font
{}                              %%% Indent amount (empty = no indent)
{\bfseries}                     %%% Thm head font
{.}                             %%% Punctuation after thm head
{ }                             %%% Space after thm head
{\thmname{#1}\thmnote{ \bfseries #3}}%%% Thm head spec
\theoremstyle{TheoremNum}
\newtheorem{reptheorem}{Theorem}
\newtheorem{repproposition}{Proposition}
\newtheorem{replemma}{Lemma}

\SetKwInput{KwInput}{Input}                % Set the Input
\SetKwInput{KwOutput}{Output}              % set the Output

\newcommand{\wh}{\widehat}
\renewcommand{\epsilon}{\varepsilon}

\renewcommand{\deg}{\text{deg}}

\newcommand{\remove}[1]{}

\begin{document}

	\title{On higher multiplicity hyperplane and polynomial covers\\for symmetry preserving subsets of the hypercube}

	\author{
		Arijit Ghosh\thanks{Indian Statistical Institute, Kolkata, India}
		\and
		Chandrima Kayal\footnotemark[1]
		\and 
		Soumi Nandi\footnotemark[1]
		\and
		S. Venkitesh\thanks{Department of Computer Science, University of Haifa, Haifa, Israel.  Webpage: \url{https://sites.google.com/view/venkitesh}\n\n Email: \email{venkitesh.mail@gmail.com}}
	}

	\date{}
	\maketitle
	
	\begin{abstract}
		Alon and F\"uredi (European J. Combin. 1993) gave a tight bound for the following hyperplane covering problem: find the minimum number of hyperplanes required to cover all points of the \(n\)-dimensional hypercube \(\{0,1\}^n\) except the origin.  Their proof is among the early instances of the \emph{polynomial method}, which considers a natural polynomial (a product of linear factors) associated to the hyperplane arrangement, and gives a lower bound on its degree, whilst being oblivious to the (product) structure of the polynomial.  Thus, their proof gives a lower bound for a \emph{weaker} polynomial covering problem, and it turns out that this bound is tight for the \emph{stronger} hyperplane covering problem.
		
		In a similar vein, solutions to some other hyperplane covering problems were obtained, via solutions of corresponding weaker polynomial covering problems, in some special cases in the works of the fourth author (Electron. J. Combin. 2022), and the first three authors (Discrete Math. 2023).  In this work, we build on these and solve a hyperplane covering problem for general symmetric sets of the hypercube, where we consider hyperplane covers with higher multiplicities.  We see that even in this generality, it is enough to solve the corresponding polynomial covering problem.  Further, this seems to be the limit of this approach as far as covering symmetry preserving subsets of the hypercube is concerned.  We gather evidence for this by considering the class of \emph{blockwise} symmetric sets of the hypercube (which is a strictly larger class than symmetric sets), and note that the same proof technique seems to only solve the polynomial covering problem.
	\end{abstract}

	%\tableofcontents
	
	\paragraph*{Notations.}  \(\mb{R}\) denotes the set of all real numbers, \(\mb{Z}\) denotes the set of all integers, \(\mb{N}\) denotes the set of all nonnegative integers, and \(\mb{Z}^+\) denotes the set of all positive integers.  \([a,b]\) denotes the closed interval of all integers between \(a\) and \(b\); further, we denote \([n]\coloneqq[1,n]\).  \(\mb{R}[\mb{X}]\) denotes the polynomial ring over the field \(\mb{R}\) and a collection of indeterminates \(\mb{X}\), where either there are \(n\) indeterminates \(\mb{X}=(X_1,\ldots,X_n)\), or there are \(N=n_1+\cdots+n_k\) indeterminates partitioned into \(k\) blocks as \(\mb{X}=(\mb{X}_1,\ldots,\mb{X}_k)\) with each \(\mb{X}_j=(X_{j,1},\ldots,X_{j,n_j})\).
	
\section{Introduction and overview}\label{sec:intro}

We will work over the field \(\mb{R}\), and consider the \(n\)-variate polynomial ring \(\mb{R}[\mb{X}]\).  A classic result by Alon and F\"uredi~\cite{alon-furedi} states that any collection of (affine) hyperplanes\footnote{We are interested in \emph{affine} hyperplanes, that is, all possible translates of codimension-1 subspaces of \(\mb{R}^n\), and not just the subspaces themselves.  However, we will suppress the adjective `affine' in the rest of this paper.} in \(\mb{R}^n\), whose union contains every point of the hypercube (or Boolean cube) \(\{0,1\}^n\) except the all-zeros point \(0^n\coloneqq(0,\ldots,0)\), must have at least \(n\) hyperplanes.  This lower bound is also tight, attained by the collection of hyperplanes defined by the equations: \(X_i=1,\,i\in[n]\).\footnote{This result of Alon and F\"uredi~\cite{alon-furedi} is, in fact, true over any field \(\mb{F}\), and not just for \(\mb{F}=\mb{R}\).}  Further, the lower bound proof by~\cite{alon-furedi} is among the early instances of the \emph{polynomial method} in combinatorics.  Note that the union of any finite collection of hyperplanes in \(\mb{R}^n\), as a set of points, is exactly equal to the zero set of the product of the affine linear polynomials defining the individual hyperplanes.  So the lower bound on the number of hyperplanes follows from a lower bound on the degree of this \emph{product polynomial}.

An interesting point to note in the lower bound proof by~\cite{alon-furedi} is that the polynomial method is oblivious to the \emph{product structure} of the polynomials corresponding to collections of hyperplanes, or \emph{any other structural property} of polynomials, and is only sensitive to the degree of the polynomials.  In other words, we may as well consider a \emph{polynomial covering problem} satisfying the same vanishing conditions -- find the minimum degree of a polynomial, \emph{among all unstructured polynomials}, that vanish at every point of \(\{0,1\}^n\) except \(0^n\) -- and the proof by the polynomial method goes through.  Therefore, in hindsight, it is amazing that the lower bound for the \emph{weaker} polynomial covering problem is, in fact, tight for the \emph{stronger} hyperplane covering problem.  In this work, we are interested in further exploring this power of the polynomial method in giving tight bounds for some hyperplane covering problems by simply considering the corresponding weaker polynomial covering problems.

In order to describe our motivations as well as our results, let us first fix some terminologies and notations.  We will identify a hyperplane \(H\) in \(\mb{R}^n\) with its defining affine linear polynomial \(H(\mb{X})\).  Let \(t\ge1,\,\ell\in[0,t-1]\), and consider any subset \(S\subsetneq\{0,1\}^n\).  We define
\begin{itemize}
	\item  a \tsf{\((t,\ell)\)-exact hyperplane cover} for \(S\) to be a finite collection of hyperplanes (considered as a multiset) in \(\mb{R}^n\) such that each point in \(S\) is contained in at least \(t\) hyperplanes, and each point in \(\{0,1\}^n\setminus S\) is contained in exactly \(\ell\) hyperplanes.
	
	\item  a \tsf{\((t,\ell)\)-exact polynomial cover} for \(S\) to be a nonzero polynomial that vanishes at each point in \(S\) with multiplicity\footnote{We say that a polynomial \(P\) vanishes at a point \(a\) with multiplicity at least \(t\) if all the derivatives of \(P\) having order at most \(t-1\) vanish at \(a\).  We will give a formal definition in the Preliminaries (Section~\ref{subsec:prelims-vanishing}).} at least \(t\), and vanishes at each point in \(\{0,1\}^n\setminus S\) with multiplicity exactly \(\ell\).
\end{itemize}
Let \(\EHC_n^{(t,\ell)}(S)\) denote the minimum size of a \((t,\ell)\)-exact hyperplane cover for \(S\), and let \(\EPC_n^{(t,\ell)}(S)\) denote the minimum degree of a \((t,\ell)\)-exact polynomial cover for \(S\).  In these notations,~\cite{alon-furedi} show the following.
\begin{theorem}[{\cite{alon-furedi}}]\label{thm:alon-furedi}
	\(\EHC_n^{(1,0)}(\{0,1\}^n\setminus\{0^n\})=\EPC_n^{(1,0)}(\{0,1\}^n\setminus\{0^n\})=n\).
\end{theorem}
\noindent It is obvious from the definitions that, in general, we have \(\EHC_n^{(t,\ell)}(S)\ge\EPC_n^{(t,\ell)}(S)\).  For completeness, we give a quick proof in Appendix~\ref{app:EHC-EPC}.  In the present work, we are broadly concerned with the following question.
\begin{question}\label{ques:main}
	Given a proper subset \(S\subsetneq\{0,1\}^n\) and integers \(t\ge1,\,\ell\in[0,t-1]\), under what conditions can we say that \(\EHC_n^{(t,\ell)}(S)=\EPC_n^{(t,\ell)}(S)\)?
\end{question} 

\subsection{Motivation}\label{subsec:motivation}

The present work could be considered a sequel to earlier works by the fourth author~\cite{venkitesh-2022-covering}, and the first three authors~\cite{ghosh-kayal-nandi-2023-covering}.  The work~\cite{venkitesh-2022-covering} relies heavily on the polynomial method using Alon's Combinatorial Nullstellensatz~\cite{alon_1999} (also see Buck, Coley, and Robbins~\cite{buck1992generalized}, and Alon and Tarsi~\cite{alon1992colorings}), and the work~\cite{ghosh-kayal-nandi-2023-covering} relies heavily on a recent \emph{multiplicity extension} of the Combinatorial Nullstellensatz given by Sauermann and Wigderson~\cite{sauermann-wigderson-2022-covering}.  The problems of concern, in the two earlier works as well as in the present work, belong to a larger class of questions that have been of interest for a long time, and have rich literature.  We mention some of these related works in Section~\ref{subsec:related}.

Let us now detail the primary motivations for our present work.
\begin{itemize}[leftmargin=*]
	\item  As a multiplicity extension of Theorem~\ref{thm:alon-furedi} for the polynomial covering problem, Sauermann and Wigderson~\cite{sauermann-wigderson-2022-covering} determined the following.
	\begin{theorem}[{\cite{sauermann-wigderson-2022-covering}}]\label{thm:sauermann-wigderson}
		For \(t\ge1,\,\ell\in[0,t-1]\), we have
		\[
		\EPC_n^{(t,\ell)}(\{0,1\}^n\setminus\{0^n\})=\begin{cases}
			n+2t-2&\tx{if}\n\ell=t-1,\\
			n+2t-3&\tx{if}\n\ell<t-1\le\big\lfloor\frac{n+1}{2}\big\rfloor.
		\end{cases}
		\]
	\end{theorem}

	\item  In a remarkable development, using techniques different from the polynomial method, Clifton and Huang~\cite{clifton2020almost} proved the following bounds for the hyperplane covering problem.
	\begin{theorem}[{\cite{clifton2020almost}}]\label{thm:clifton-huang}
		For all \(n\ge3,\,t\ge4\), we have
		\[
		n+t+1\le\EHC_n^{(t,0)}(\{0,1\}^n\setminus\{0^n\})\le n+\binom{t}{2},
		\]
		Further, for \(n\ge2,\,t=2,3\), we have \(\EHC_n^{(t,0)}(\{0,1\}^n\setminus\{0^n\})=n+\displaystyle\binom{t}{2}\).
	\end{theorem}

	\item  We say a subset \(S\subseteq\{0,1\}^n\) is \tsf{symmetric} if \(S\) is closed under permutations of coordinates.  Note that the \tsf{Hamming weight} of any \(x\in\{0,1\}^n\) is defined by \(|x|=|\{i\in[n]:x_i=1\}|\).  Thus, the subset \(S\) is symmetric if and only if
	\[
	x\in S,\,y\in\{0,1\}^n,\,|y|=|x|\quad\implies\quad y\in S.
	\]
	For any symmetric set \(S\subseteq\{0,1\}^n\), we define \(W_n(S)=\{|x|:x\in S\}\).  It is immediate that a symmetric set \(S\) is determined by the corresponding set \(W_n(S)\).  Also for \(i\in[0,n]\), let \(W_{n,i}=[0,i-1]\cup[n-i+1,n]\), and define the symmetric set \(T_{n,i}\subseteq\{0,1\}^n\) by \(W_n(T_{n,i})=W_{n,i}\).\footnote{Here we have \(W_{n,0}=\emptyset\) and \(T_{n,0}=\emptyset\).}
	
	The fourth author~\cite{venkitesh-2022-covering} gave a combinatorial characterization of \(\EPC_n^{(1,0)}(S)\) for all symmetric sets \(S\subsetneq\{0,1\}^n\), as well as a partial result towards answering Question~\ref{ques:main} in this setting.  The characterization is in terms of a simple combinatorial measure.  For any symmetric set \(S\subseteq\{0,1\}^n\), define
	\begin{align*}
	\mu_n(S)&=\max\{i\in[0,\lceil n/2\rceil]:W_{n,i}\subseteq W_n(S)\},\\
	\tx{and}\quad\Lambda_n(S)&=|W_n(S)|-\mu_n(S).
	\end{align*}
	Further, denote \(\ol{\mu}_n(S)\coloneqq\mu_n(\{0,1\}^n\setminus S)\) and \(\ol{\Lambda}_n(S)\coloneqq\Lambda_n(\{0,1\}^n\setminus S)\).
	
	\begin{theorem}[{\cite{venkitesh-2022-covering}}]\label{thm:venkitesh}
		\begin{enumerate}[(a)]
			\item  For any symmetric set \(S\subsetneq\{0,1\}^n\), we have
			\[
			\EPC_n^{(1,0)}(S)=\Lambda_n(S).
			\]
			
			\item  For any symmetric set \(S\subsetneq\{0,1\}^n\) such that \(W_{n,2}\not\subseteq W_n(S)\), we have
			\begin{align*}
				\EHC_n^{(1,0)}(S)=\EPC_n^{(1,0)}(S)&=\Lambda_n(S)\\
				&=\begin{cases}
					|W_n(S)|&\tx{if}\n W_{n,1}\not\subseteq W_n(S),\\
					|W_n(S)|-1&\tx{if}\n W_{n,1}\subseteq W_n(S).
				\end{cases}
			\end{align*}
			
			\item  \(\EHC_n^{(1,0)}(T_{n,2})=\EPC_n^{(1,0)}(T_{n,2})=2=|W_{n,2}|-\mu_n(T_{n,2})=\Lambda_n(T_{n,2})\).
		\end{enumerate}
	\end{theorem}
	It is interesting, and important for further discussions, to note the constructions that imply the equalities in Theorem~\ref{thm:venkitesh}.
	\begin{example}[{\cite{venkitesh-2022-covering}}]\label{ex:venkitesh}
		\begin{enumerate}[(a)]
			\item  Let \(S\subsetneq\{0,1\}^n\) be a symmetric set.  By the proof of Theorem~\ref{thm:venkitesh}(a)~\cite[Proposition 6.1]{venkitesh-2022-covering}, for every \(a\in\{0,1\}^n\setminus S\), there exists a polynomial \(Q_a(\mb{X})\in\mb{R}[\mb{X}]\) such that \(\deg(Q_a)\le\Lambda_n(S)\), \(Q_a|_S=0\), and \(Q_a(a)=1\).  Then choose scalars \(\beta_a\in\mb{R},\,a\in\{0,1\}^n\setminus S\) such that the polynomial \(Q(\mb{X})\coloneqq\sum_{a\in\{0,1\}^n\setminus S}\beta_aQ_a(\mb{X})\) satisfies \(\deg(Q)\le\Lambda_n(S)\), \(Q|_S=0\), and \(Q(b)\ne0\) for all \(b\in\{0,1\}^n\setminus S\).  So the polynomial \(Q(\mb{X})\) witnesses the equality in Theorem~\ref{thm:venkitesh}(a).
			
			Note that the set of scalars \(B\coloneqq\{\beta_a:a\in\{0,1\}^n\setminus S\}\) can always be chosen so that \(Q(\mb{X})\) satisfies the above required conditions.  For instance, consider a subfield of \(\mb{R}\) defined by \(\widehat{\mb{Q}}\coloneqq\mb{Q}\big(\{Q_a(b):a,b\in\{0,1\}^n\setminus S\}\big)\).\footnote{For any \(B\subsetneq\mb{R}\), the notation \(\mb{Q}(B)\) denotes the smallest subfield of \(\mb{R}\) that contains \(\mb{Q}\) and \(B\).  This subfield exists and is unique, by elementary field theory.}  It then follows that \(\mb{R}\) is an infinite dimensional \(\widehat{\mb{Q}}\)-vector space.  So we can choose \(B\) to be any \(\widehat{\mb{Q}}\)-linearly independent subset of \(\mb{R}\) of size \(2^n-|S|\).
			
			\item  Let \(S\subsetneq\{0,1\}^n\) be a symmetric set such that \(W_{n,2}\not\subseteq W_n(S)\).  If \(W_{n,1}\not\subseteq W_n(S)\), then the collection of hyperplanes \(\{H'_t(\mb{X}):t\in W_n(S)\}\), defined by \(H'_t(\mb{X})\coloneqq\sum_{i=1}^nX_i-t,\,t\in W_n(S)\) witnesses equality in Theorem~\ref{thm:venkitesh}(b).  If \(W_{n,1}=\{0,n\}\subseteq W_n(S)\), note that the hyperplane \(H^*_{(1,1)}(\mb{X})\coloneqq\sum_{i=1}^{n-1}X_i-(n-1)X_n\) satisfies \(H^*_{(1,1)}(x)=0\) for \(x\in\{0,1\}^n\) if and only if \(x\in\{0^n,1^n\}\), that is, \(x\in T_{n,1}\).  Then the collection of hyperplanes \(\{H^*_{(1,1)}(\mb{X})\}\sqcup\{H'_t(\mb{X}):t\in W_n(S)\setminus\{0,n\}\}\) witnesses the equality in Theorem~\ref{thm:venkitesh}(b).
			
			\item  The collection of hyperplanes \(\{H^*_{(2,1)}(\mb{X}),H^*_{(2,2)}(\mb{X})\}\), where \(H^*_{(2,1)}(\mb{X})\coloneqq\sum_{i=1}^{n-1}X_i-(n-3)X_n+1\) and \(H^*_{(2,2)}(\mb{X})\coloneqq\sum_{i=1}^{n-2}X_i-(n-2)X_{n-1}\), witnesses the equality in Theorem~\ref{thm:venkitesh}(c).
		\end{enumerate}
	\end{example}

	Further, the following was conjectured in~\cite{venkitesh-2022-covering}, appealing to Theorem~\ref{thm:venkitesh}(b) and (c).
	\begin{conjecture}[{\cite{venkitesh-2022-covering}}]\label{conj:venkitesh-EHC}
		For any symmetric set \(S\subsetneq\{0,1\}^n\) such that \(W_{n,2}\subseteq W_n(S)\), we have
		\[
		\EHC_n^{(1,0)}(S)=|W_n(S)|-2,
		\]
		and therefore, \(\EHC_n^{(1,0)}(S)>\EPC_n^{(1,0)}(S)\) if \(W_{n,2}\subsetneq W_n(S)\).
	\end{conjecture}
	
	\item  Aaronson, Groenland, Grzesik, Kielak, and Johnston~\cite{aaronson2020exact} considered problem of determining \(\EHC_n^{(1,0)}(\{0,1\}^n\setminus S)\) for general nonempty subsets \(S\subseteq\{0,1\}^n\), and obtained the following.
	\begin{theorem}[{\cite{aaronson2020exact}}]\label{thm:aaronson}
		For any nonempty subset \(S\subseteq\{0,1\}^n\), we have
		\[
		\EHC_n^{(1,0)}(\{0,1\}^n\setminus S)\ge n-\lfloor\log_2|S|\rfloor.
		\]
	\end{theorem}
	Improving upon Theorem~\ref{thm:aaronson},~\cite{ghosh-kayal-nandi-2023-covering} bounded \(\EPC_n^{(t,t-1)}(\{0,1\}^n\setminus S)\) for all \(t\ge1\), in a more abstract sense by introducing a combinatorial measure called \emph{index complexity}.  For any subset \(S\subseteq\{0,1\}^n,\,|S|>1\), the \tsf{index complexity} of \(S\) is defined to be the smallest positive integer \(r_n(S)\) such that for some \(I\subseteq[n],\,|I|=r_n(S)\), there is a point \(u\in S\) such that for each \(v\in S,\,v\ne u\), we get \(v_i\ne u_i\) for some \(i\in I\), that is, the point \(u\) can be \emph{separated from all other points} in \(S\) in the coordinates in \(I\). (The index complexity of a singleton set is defined to be zero.)
	
	The improvement to Theorem~\ref{thm:aaronson} was achieved via the following two results.
	\begin{proposition}[{\cite{ghosh-kayal-nandi-2023-covering}}]\label{pro:index-complexity}
		For any nonempty subset \(S\subseteq\{0,1\}^n\), we have
		\[
		r_n(S)\le\lfloor\log_2|S|\rfloor.
		\]
	\end{proposition}
	
	\begin{theorem}[{\cite{ghosh-kayal-nandi-2023-covering}}]\label{thm:EPC-index}
		For any nonempty subset \(S\subseteq\{0,1\}^n\) and \(t\ge1\), we have
		\[
		\EPC_n^{(t,t-1)}(\{0,1\}^n\setminus S)\ge n-r_n(S)+2t-2.
		\]
	\end{theorem}
	
	\item  Returning to the context of symmetric sets, note that for any symmetric set \(S\subseteq\{0,1\}^n\), the complement \(\{0,1\}^n\setminus S\) is also symmetric.  Further, we say a symmetric set \(S\) is a \tsf{layer} if \(|W_n(S)|=1\).  The first three authors~\cite{ghosh-kayal-nandi-2023-covering} answered Question~\ref{ques:main} in the affirmative for the complement of any layer \(S\), and for all \(t\ge1,\,\ell=t-1\).  In particular, this improves Theorem~\ref{thm:sauermann-wigderson}.
	\begin{theorem}[{\cite{ghosh-kayal-nandi-2023-covering}}]\label{thm:ghosh-kayal-nandi}
		For any layer \(S\subsetneq\{0,1\}^n\) with \(W_n(S)=\{w\}\), and any \(t\ge1\), we have
		\[
		\EHC_n^{(t,t-1)}(\{0,1\}^n\setminus S)=\EPC_n^{(t,t-1)}(\{0,1\}^n\setminus S)=\max\{w,n-w\}+2t-2.
		\]
	\end{theorem}
	The following construction by~\cite{ghosh-kayal-nandi-2023-covering} is important throughout our discussion.  For completeness, we give a proof in Appendix~\ref{app:GKN-hyperplane}.
	\begin{lemma}[{\cite{ghosh-kayal-nandi-2023-covering}}]\label{lem:T}
		For \(i\in[0,\lceil n/2\rceil]\), the collection of hyperplanes \(\{H^*_{(i,j)}(\mb{X}):j\in[i]\}\) defined by
		\[
		H^*_{(i,j)}(\mb{X})=\sum_{k=1}^{n-j}X_k-(n-2i+j)X_{n-j+1}-(i-j),\quad j\in[i],
		\]
		satisfies the following.
		\begin{itemize}
			\item[\(\bullet\)]  For every \(a\in T_{n,i}\), there exists \(j\in[i]\) such that \(H^*_{(i,j)}(a)=0\).
			\item[\(\bullet\)]  \(H^*_{(i,j)}(b)\ne0\) for every \(b\in\{0,1\}^n\setminus T_{n,i},\,j\in[i]\).
		\end{itemize}
	\end{lemma}
	A construction that implies the equality in Theorem~\ref{thm:ghosh-kayal-nandi} is then immediate.
	\begin{example}[{\cite{ghosh-kayal-nandi-2023-covering}}]\label{ex:GKN-EHC}
		Let \(S\subsetneq\{0,1\}^n\) be a layer with \(W_n(S)=w\), and \(t\ge1\).  Let \(w'=\min\{w,n-w\}\).  Denote \(H_0^\circ(\mb{X})=X_1,\,H_1^\circ(\mb{X})=X_1-1\).  Then the collection of hyperplanes
		\begin{align}
		\{H^*_{(w',j)}(\mb{X}):j\in[w']\}\n\sqcup\bigsqcup_{\ell\in[t-1]}\{H_0^\circ(\mb{X}),H_1^\circ(\mb{X})\}\tag*{(disjoint union, as a multiset)}
		\end{align}
		witnesses the equality in Theorem~\ref{thm:ghosh-kayal-nandi}.
	\end{example}
	\begin{remark}\label{rem:EPC-index}
		Appealing to Theorem~\ref{thm:venkitesh}, Theorem~\ref{thm:ghosh-kayal-nandi}, and the definition of index complexity, it will be interesting ahead to note that for a layer \(S\subsetneq\{0,1\}^n\) with \(W_n(S)=w\), we have\footnote{We will understand the index complexity of symmetric sets in more detail in Section~\ref{subsec:main-index-complexity-symmetric}.}
		\[
		\ol{\Lambda}_n(S)=\max\{w,n-w\}=n-r_n(S),
		\]
		and Theorem~\ref{thm:ghosh-kayal-nandi}, in fact, shows that for the layer \(S\) and for all \(t\ge1\), we have
		\begin{align*}
			\EHC_n^{(t,t-1)}(\{0,1\}^n\setminus S)&=\EHC_n^{(1,0)}(\{0,1\}^n\setminus S)+2t-2\\
			&=\ol{\Lambda}_n(S)+2t-2\\
			&=n-r_n(S)+2t-2\\
			&=\EPC_n^{(1,0)}(\{0,1\}^n\setminus S)+2t-2=\EPC_n^{(t,t-1)}(\{0,1\}^n\setminus S).
		\end{align*}
	\end{remark}
	Further,~\cite{ghosh-kayal-nandi-2023-covering} disprove Conjecture~\ref{conj:venkitesh-EHC}, pertaining to the remaining case `\(W_{n,2}\subsetneq W_n(S)\)', by providing a counterexample.
\end{itemize}
In the present work, we will build upon some of the above results.

\subsection{Our results: higher multiplicity hyperplane covers}\label{subsec:results-hyperplane}

As mentioned earlier, we are broadly interested in understanding when Question~\ref{ques:main} has an answer in the affirmative.  In the present work, we will obtain some such characterizations when \(t\ge1,\,\ell=t-1\), for some structured subsets of the hypercube; specifically, we will consider symmetric sets, as well as a \emph{block generalization} of symmetric sets.  Strictly speaking, we will also have some nondegeneracy conditions in some characterizations.

\subsubsection*{Proof technique}\label{subsubsec:technique}
\addcontentsline{toc}{subsubsection}{Proof technique}

We also have a common proof technique for our results, which is simple and similar to the approach adopted in the earlier works~\cite{alon-furedi,sauermann-wigderson-2022-covering,venkitesh-2022-covering,ghosh-kayal-nandi-2023-covering}.  To summarize the technique, consider a subset \(S\subsetneq\{0,1\}^n\) (with a suitable structure, as we detail later), and suppose we would like to determine \(\EHC_n^{(t,t-1)}(S)\).  Via the polynomial method, we first obtain a lower bound for the \emph{weaker} polynomial covering problem, say \(\EPC_n^{(t,t-1)}(S)\ge L_t\) (for some \(L_t\ge1\)).  We then construct a hyperplane cover to obtain an upper bound \(\EHC_n^{(t,t-1)}(S)\le L_t\) for the \emph{stronger} hyperplane covering problem.  Thus, we immediately have the inequalities
\[
L_t\ge\EHC_n^{(t,t-1)}(S)\ge\EPC_n^{(t,t-1)}(S)\ge L_t,
\]
which gives a tight characterization.

\subsubsection*{Some fundamental hyperplane families}\label{subsubsec:hyperplane-construction}
\addcontentsline{toc}{subsubsection}{Some fundamental hyperplane families}

Before we detail our results, let us fix the notations for some fundamental hyperplane families which will appear repeatedly in this work.
\begin{enumerate}[(a)]
	\item  For each \(t\in[0,n]\), define \(H'_t(\mb{X})=\sum_{i=1}^nX_i-t\).  Further, for any \(W\subseteq[0,n]\), let \(\mc{H}'_W(\mb{X})=\{H'_t(\mb{X}):t\in W\}\).
	
	\item  For each \(i\in[0,\lceil n/2\rceil],\,j\in[i]\), as defined in Lemma~\ref{lem:T}, we have
	\[
	H^*_{(i,j)}(\mb{X})=\sum_{k=1}^{n-j}X_k-(n-2i+j)X_{n-j+1}-(i-j).
	\]
	Further, let \(\mc{H}^*_i(\mb{X})=\{H^*_{(i,j)}(\mb{X}):j\in[i]\}\).
	
	\item  Define \(H_0^\circ(\mb{X})=X_1\) and \(H_1^\circ(\mb{X})=X_1-1\).  Further, let \(\mc{H}^{\circ m}(\mb{X})=\bigsqcup_{\,\ell=1}^{\,m}\{H^\circ_0(\mb{X}),H_1^\circ(\mb{X})\}\) (disjoint union, as a multiset), for any \(m\ge1\).
\end{enumerate}

\subsubsection{Warm-up: Index complexity of symmetric sets}\label{subsubsec:warmup}

We have seen that~\cite{ghosh-kayal-nandi-2023-covering} obtain a lower bound on the polynomial covering problem for any general subset of the hypercube, in terms of index complexity (Theorem~\ref{thm:EPC-index}), by employing the polynomial method.  Further, they show that in the case of a single layer of the hypercube, this lower bound is tight (Theorem~\ref{thm:ghosh-kayal-nandi}).  As a consequence, it is also seen that the index complexity of a single layer can be expressed in terms of the combinatorial measure \(\Lambda_n\) introduced in~\cite{venkitesh-2022-covering} (also see Theorem~\ref{thm:venkitesh} and Remark~\ref{rem:EPC-index}).  To summarise, we have the following.
\begin{proposition}[Implicit in~\cite{ghosh-kayal-nandi-2023-covering}]\label{pro:layer-index-complexity}
	For a layer \(S\subsetneq\{0,1\}^n\) with \(W_n(S)=\{w\}\), we have
	\[
	\ol{\Lambda}_n(S)=n-r_n(S)=\max\{w,n-w\}.
	\]
\end{proposition}
\noindent Such an equality is no longer true for general symmetric sets.  We can, in fact, precisely understand the general case combinatorially.  We introduce some terminology before we proceed.

For any \(a\in[-1,n-1],\,b\in[1,n+1],\,a<b\), denote the set of weights \(I_{n,a,b}=[0,a]\cup[b,n]\), and we say a \tsf{peripheral interval} is the symmetric set \(J_{n,a,b}\subseteq\{0,1\}^n\) defined by \(W_n(J_{n,a,b})=I_{n,a,b}\).\footnote{Here, we have the convention \([0,-1]=[n+1,n]=\emptyset\).}  We will consider \emph{inner and outer approximations} of a symmetric set.

Let \(S\subseteq\{0,1\}^n\) be a symmetric set.
\begin{itemize}
	\item  If \(S\subsetneq\{0,1\}^n\), then the \tsf{inner interval} of \(S\), denoted by \(\innint(S)\), is defined to be the peripheral interval \(J_{n,a,b}\subseteq\{0,1\}^n\) of maximum size such that \(J_{n,a,b}\subseteq S\).  Further, we define \(\innint(\{0,1\}^n)=J_{n,\lfloor n/2\rfloor,\lfloor n/2\rfloor+1}\).
	
	\item  Let \(\mc{O}(S)\) be the collection of all peripheral intervals \(J_{n,a,b}\) such that \(S\subseteq J_{n,a,b}\) and \(I_{n,a,b}=W_n(J_{n,a,b})\) has minimum size.  It is easy to check that there exists either a unique peripheral interval \(J_{n,a,b}\in\mc{O}(S)\), or exactly a pair of peripheral intervals \(J_{n,a,b},J_{n,n-b,n-a}\in\mc{O}(S)\) such that the quantity \(|a+b-n|\) is minimum.  The \tsf{outer interval} of \(S\), denoted by \(\outint(S)\), is defined by
	\[
	\outint(S)=\begin{cases}
		J_{n,a,b}&\tx{if }J_{n,a,b}\tx{ is the unique minimizer of }|a+b-n|,\\
		J_{n,a,b}&\tx{if }J_{n,a,b},J_{n,n-b,n-a}\tx{ are minimizers of }|a+b-n|,\tx{ and }a>n-b.
	\end{cases}
	\]
\end{itemize}
We will elaborate a bit on the definitions in the Preliminaries (Section~\ref{subsec:prelims-inner-outer-peripheral}), and further discuss the uniqueness (and therefore, well-definedness) of inner and outer intervals, along with some illustrations, in Appendix~\ref{app:inner-outer}.  Now define
\begin{align*}
	\inn_n(S)&=(\min\{a,n-b\}+1)+|W_n(S)\setminus W_{n,\min\{a,n-b\}+1}|&&\tx{where }J_{n,a,b}=\innint(S),\\
	\tx{and }\out_n(S)&=a+n-b+1=|I_{n,a,b}|-1&&\tx{where }J_{n,a,b}=\outint(S).
\end{align*}

Towards understanding the index complexity of general symmetric sets, we obtain the following important relation between inner and outer intervals of symmetric sets.
\begin{proposition}\label{pro:inner-outer}
	For any nonempty symmetric set \(S\subseteq\{0,1\}^n\), we have
	\[
	\inn_n(\{0,1\}^n\setminus S)+\out_n(S)\ge n.
	\]
	Further, equality holds if and only if either \(S\) or \(\{0,1\}^n\setminus S\) is a peripheral interval.
\end{proposition}
\noindent  We are now ready to characterize the index complexity of symmetric sets.
\begin{proposition}\label{pro:index-complexity-symmetric}
	For any nonempty symmetric set \(S\subseteq\{0,1\}^n\), we have \(r_n(S)=\out_n(S)\).
\end{proposition}
\noindent Also, the following is trivial, by definitions.
\begin{fact}[By definitions]\label{fact:lambda}
	For any symmetric set \(S\subsetneq\{0,1\}^n\), we have \(\Lambda_n(S)=\inn_n(S)\).
\end{fact}
\noindent The following is then an immediate corollary of Proposition~\ref{pro:inner-outer}, Proposition~\ref{pro:index-complexity-symmetric}, and Fact~\ref{fact:lambda}.
\begin{corollary}\label{cor:lambda-index-complexity}
	For any nonempty symmetric set \(S\subseteq\{0,1\}^n\), we have
	\[
	\ol{\Lambda}_n(S)\ge n-r_n(S).
	\]
	Further, equality holds if and only if either \(S\) or \(\{0,1\}^n\setminus S\) is a peripheral interval.
\end{corollary}
\noindent  Note that if \(S\subseteq\{0,1\}^n\) is a layer, then \(\{0,1\}^n\setminus S\) is a peripheral interval, and hence Corollary~\ref{cor:lambda-index-complexity} recovers Proposition~\ref{pro:layer-index-complexity}.

\subsubsection{Covering symmetric sets}\label{subsubsec:covering-symmetric}

  Note that~\cite{ghosh-kayal-nandi-2023-covering} had disproved Conjecture~\ref{conj:venkitesh-EHC}, but not characterized \(\EHC_n^{(1,0)}(S)\) for all symmetric sets \(S\subsetneq\{0,1\}^n\).  We obtain this characterization here, and more, with our first main result.  (In particular, this answers a question of the fourth author~\cite[Open Problem 36]{venkitesh-2022-covering}.)  Our first main result extends Theorem~\ref{thm:venkitesh} and Theorem~\ref{thm:ghosh-kayal-nandi}, and answers Question~\ref{ques:main} in the affirmative for symmetric sets, with \(t\ge1,\,\ell=t-1\).  As a proof attempt, for a general symmetric set, we may directly apply Theorem~\ref{thm:EPC-index} (which was obtained in~\cite{ghosh-kayal-nandi-2023-covering} by the polynomial method), and then attempt to find a tight construction.  This would require a precise understanding of the index complexity of symmetric sets, which we obtain in Proposition~\ref{pro:index-complexity-symmetric}.  However, the lower bound obtained Thus, is \emph{weak}.  It turns out that the tight lower bound is larger, and the gap is, in fact, exactly captured by Corollary~\ref{cor:lambda-index-complexity}!

For convenience, we will state the result in terms of complements of symmetric sets (which are also symmetric).  This will be an important distinction in an extended setting, which we consider later.
\begin{theorem}\label{thm:multiplicity-symmetric}
	For any nonempty symmetric set \(S\subseteq\{0,1\}^n\) and \(t\ge1\), we have
	\[
	\EHC_n^{(t,t-1)}(\{0,1\}^n\setminus S)=\EPC_n^{(t,t-1)}(\{0,1\}^n\setminus S)=\ol{\Lambda}_n(S)+2t-2.
	\]
\end{theorem}
\noindent Interestingly, we obtain the tight bound in Theorem~\ref{thm:multiplicity-symmetric} since our instantiation of the polynomial method turns out to be \emph{stronger} than that in the proof of Theorem~\ref{thm:EPC-index} by~\cite{ghosh-kayal-nandi-2023-covering}.  This relative strength is also captured exactly by Corollary~\ref{cor:lambda-index-complexity}!
\begin{remark}\label{rem:first-EHC-equivalent}
	The proof of Theorem~\ref{thm:multiplicity-symmetric} will, in fact, show that for any nonempty symmetric set \(S\subseteq\{0,1\}^n\) and \(t\ge1\), we have
	\begin{align*}
		\EHC_n^{(t,t-1)}(\{0,1\}^n\setminus S)&=\EHC_n^{(1,0)}(\{0,1\}^n\setminus S)+2t-2\\
		&=\ol{\Lambda}_n(S)+2t-2\\
		&=\EPC_n^{(1,0)}(\{0,1\}^n\setminus S)+2t-2=\EPC_n^{(t,t-1)}(\{0,1\}^n\setminus S).
	\end{align*}
\end{remark}
\noindent A simple generalization of Example~\ref{ex:GKN-EHC} gives a construction implying the equality in Theorem~\ref{thm:multiplicity-symmetric}.
\begin{example}\label{ex:multi-symm}
	Let \(S\subseteq\{0,1\}^n\) be a nonempty symmetric set, and \(t\ge1\).  Then the collection of hyperplanes
	\[
	\mc{H}^*_{\ol{\mu}_n(S)}(\mb{X})\sqcup\mc{H}'_{W_n(\{0,1\}^n\setminus S)\setminus W_{n,\ol{\mu}_n(S)}}(\mb{X})\sqcup\mc{H}^{\circ(t-1)}(\mb{X})
	\]
	witnesses the equality in Theorem~\ref{thm:multiplicity-symmetric}.
\end{example}

\subsubsection{Covering special \emph{\(k\)-wise symmetric} sets}\label{subsubsec:covering-k-symmetric}

Fix a positive integer \(k\ge1\), and consider the hypercube \(\{0,1\}^N\) as a product of \(k\) hypercubes \(\{0,1\}^N=\{0,1\}^{n_1}\times\cdots\times\{0,1\}^{n_k}\) (and so \(N=n_1+\cdots+n_k\)).  We would like to extend the notion of symmetric sets to subsets in \(\{0,1\}^N\) which also respect the structure of \(\{0,1\}^N\) as a \emph{product of \(k\) blocks}.  We define a subset \(S\subseteq\{0,1\}^N\) to be a \tsf{\(k\)-wise grid} if \(S=S_1\times\cdots\times S_k\), where each \(S_i\subseteq\{0,1\}^{n_i}\) is symmetric.  Further, we say \(S=S_1\times\cdots\times S_k\) is a \tsf{\(k\)-wise layer} if each \(S_i\) is a layer.  Then we define a general \tsf{\(k\)-wise symmetric set} to be a union of an arbitrary collection of \(k\)-wise layers.

Note that every \(k\)-wise grid \(S_1\times\cdots\times S_k\) is a \(k\)-wise symmetric set, as given by
\[
S=\bigsqcup_{\tx{layer }L_i\subseteq S_i,\,i\in[k]}(L_1\times\cdots\times L_k),
\]
but the converse is not true.  For instance, the complement of a \(k\)-wise layer \(L_1\times\cdots\times L_k\) is \(k\)-wise symmetric, as given by
\[
\{0,1\}^N\setminus(L_1\times\cdots\times L_k)=\bigsqcup_{\substack{\emptyset\ne I\subseteq[k]\\\tx{layer }\wt{L}_i\subseteq\{0,1\}^{n_i},\,\wt{L}_i\ne L_i,\,i\in I\\\tx{layer }L'_i\subseteq\{0,1\}^{n_i},\,i\not\in I\phantom{,\,\wt{L}_i\ne L_i}}}\bigg(\prod_{i\in I}\wt{L}_i\bigg)\times\bigg(\prod_{i\not\in I}L'_i\bigg),
\]
which is clearly not a \(k\)-wise grid.

\subsubsection*{Covering complements of \(k\)-wise grids}\label{subsubsec:covering-k-grid-complement}

Our second main result extends Theorem~\ref{thm:multiplicity-symmetric} to complements of \(k\)-wise grids, Thus, answering Question~\ref{ques:main} in the affirmative in this case.
\begin{theorem}\label{thm:multiplicity-block-symmetric}
	For any nonempty \(k\)-wise grid \(S=S_1\times\cdots\times S_k\subseteq\{0,1\}^N\) and \(t\ge1\), we have
	\[
	\EHC_N^{(t,t-1)}(\{0,1\}^N\setminus S)=\EPC_N^{(t,t-1)}(\{0,1\}^N\setminus S)=\sum_{i=1}^k \ol{\Lambda}_{n_i}(S_i)+2t-2.
	\]
\end{theorem}
\begin{remark}\label{rem:second-EHC-equivalent}
	The proof of Theorem~\ref{thm:multiplicity-block-symmetric} will, in fact, show that for any nonempty \(k\)-wise grid \(S=S_1\times\cdots\times S_k\subseteq\{0,1\}^N\) and \(t\ge1\), we have
	\begin{align*}
		\EHC_N^{(t,t-1)}(\{0,1\}^N\setminus S)&=\EHC_N^{(1,0)}(\{0,1\}^N\setminus S)+2t-2\\
		&=\sum_{i=1}^k\ol{\Lambda}_{n_i}(S_i)+2t-2\\
		&=\EPC_N^{(1,0)}(\{0,1\}^N\setminus S)+2t-2=\EPC_N^{(t,t-1)}(\{0,1\}^N\setminus S).
	\end{align*}
\end{remark}
\noindent A construction that implies the equality in Theorem~\ref{thm:multiplicity-block-symmetric} is a block extension of Example~\ref{ex:multi-symm}.
\begin{example}\label{ex:multi-block-symm}
	Let \(S=S_1\times\cdots\times S_k\subseteq\{0,1\}^N\) be a nonempty \(k\)-wise grid, and \(t\ge1\).  Then the collection of hyperplanes
	\[
	\bigg(\bigsqcup_{j=1}^k\big(\mc{H}^*_{\ol{\mu}_{n_j}(S_j)}(\mb{X}_j)\sqcup\mc{H}'_{W_{n_j}(\{0,1\}^{n_j}\setminus S_j)\setminus W_{n_j,\ol{\mu}_{n_j}(S_j)}}(\mb{X}_j)\big)\bigg)\sqcup\mc{H}^{\circ(t-1)}(\mb{X}_1)
	\]
	witnesses the equality in Theorem~\ref{thm:multiplicity-block-symmetric}.
\end{example}

\subsubsection*{A special case: covering subcubes and their complements}\label{subsubsec:covering-subcube}

Here we consider the special case of \(2\)-wise grids, where one of the blocks is a \emph{full} hypercube.  The results we mention here are immediate from previous results, and hence we simply mention them without repeating the proofs.

By a \tsf{subcube} of a hypercube \(\{0,1\}^n\), we mean a subset of the form \(\{0,1\}^I\times\{a\}\), where \(I\subseteq[n]\) and \(a\in\{0,1\}^{[n]\setminus I}\).  Since we are now concerned with polynomials with vanishing conditions on a subcube, without loss of generality, we will assume that the subcube is \(\mc{Q}_m\coloneqq\{0,1\}^m\times\{0^{n-m}\}\), for some \(m\in[0,n]\).  This is true since we can permute coordinates, as well as introduce translations of variables in any polynomial without changing the degree of the polynomial.  Further, we will assume that \(1\le m\le n-1\).  So \(\mc{Q}_m\) is a 2-wise grid, where we consider the product \(\{0,1\}^n=\{0,1\}^m\times\{0,1\}^{n-m}\).

\paragraph*{Covering complements of subcubes.}  As a consequence of Theorem~\ref{thm:multiplicity-block-symmetric}, we immediately get the following about covering complements of subcubes.
\begin{corollary}\label{cor:subcube-complement}
	For any \(1\le m\le n-1\) and \(t\ge1\), we have
	\[
	\EHC_n^{(t,t-1)}(\{0,1\}^n\setminus \mc{Q}_m)=\EPC_n^{(t,t-1)}(\{0,1\}^n\setminus \mc{Q}_m)=n-m+2t-2.
	\]
\end{corollary}
\noindent  In this case, Example~\ref{ex:multi-block-symm} simplifies to the following.
\begin{example}\label{ex:subcube-complement}
	Let \(1\le m\le n-1\) and \(t\ge1\).  Then the collection of hyperplanes
	\[
	\{X_{m+1}-1,\ldots,X_n-1\}\sqcup\mc{H}^{\circ(t-1)}(\mb{X})
	\]
	witnesses the equality in Corollary~\ref{cor:subcube-complement}.
\end{example}

  A variant of Corollary~\ref{cor:subcube-complement} can be obtained by considering arbitrary symmetric sets in the \emph{second block}.  For a symmetric set \(S\subseteq\{0,1\}^{n-m}\), denote \(\mc{Q}_m(S)=\{0,1\}^m\times S\).
\begin{corollary}\label{cor:subcube-symmetric}
	For \(1\le m\le n-1\), any nonempty symmetric set \(S\subseteq\{0,1\}^{n-m}\), and \(t\ge1\), we have
	\[
	\EHC_n^{(t,t-1)}(\{0,1\}^n\setminus \mc{Q}_m(S))=\EPC_n^{(t,t-1)}(\{0,1\}^n\setminus \mc{Q}_m(S))=\ol{\Lambda}_{n-m}(S)+2t-2.
	\]
\end{corollary}
\noindent  In this case, Example~\ref{ex:multi-block-symm} simplifies to the following.
\begin{example}\label{ex:subcube-symmetric}
	Let \(1\le m\le n-1\), \(S\subseteq\{0,1\}^{n-m}\) be a nonempty symmetric set, and \(t\ge1\).  Also denote \(\mb{X}=(\mb{X}',\mb{X}'')\) with \(\mb{X}'=(X_1,\ldots,X_m),\,\mb{X}''=(X_{m+1},\ldots,X_n)\).  Then the collection of hyperplanes
	\[
	\mc{H}^*_{\ol{\mu}_{n-m}(S)}(\mb{X}'')\sqcup\mc{H}'_{W_{n-m}(\{0,1\}^{n-m}\setminus S)\setminus W_{n-m,\ol{\mu}_{n-m}(S)}}(\mb{X}'')\sqcup\mc{H}^{\circ(t-1)}(\mb{X}'')
	\]
	witnesses the equality in Corollary~\ref{cor:subcube-symmetric}.
\end{example}

\paragraph*{Covering subcubes.}  It turns out that covering subcubes is easier than covering their complements.  In fact, we can even consider a more general case -- with an arbitrary symmetric set in the second block, as in Corollary~\ref{cor:subcube-symmetric}, as well as more general multiplicities.  We give a quick proof here.
\begin{proposition}\label{pro:2-wise-subcube}
	For \(1\le m\le n-1\), any symmetric set \(S\subseteq\{0,1\}^{n-m}\), and \(t\ge1,\,\ell\in[0,t-1]\), we have
	\[
	\EHC_n^{(t,\ell)}(\{0,1\}^m\times S)=\EHC_{n-m}^{(t,\ell)}(S).
	\]
	In particular, for any nomempty symmetric set \(S\subseteq\{0,1\}^{n-m}\) and \(t\ge1\), we have
	\[
	\EHC_n^{(t,t-1)}(\{0,1\}^m\times S)=\EHC_{n-m}^{(t,t-1)}(S)=\Lambda_{n-m}(S)+2t-2.
	\]
\end{proposition}
\begin{proof}
	Denote the indeterminates \(\mb{X}=(\mb{X}',\mb{X}'')\) with \(\mb{X}'=(X_1,\ldots,X_m),\,\mb{X}''=(X_{m+1},\ldots,X_n)\).  Let \(\mc{H}(\mb{X})=\{h_1(\mb{X}),\ldots,h_q(\mb{X})\}\) be a \((t,\ell)\)-exact hyperplane cover for \(\{0,1\}^m\times S\) with \(q=|\mc{H}|=\EHC_n^{(t,\ell)}(\{0,1\}^m\times S)\).  Now let \(\mc{H}''(\mb{X}'')=\mc{H}(0^m,\mb{X}'')=\{h_1(0^m,\mb{X}''),\ldots,h_q(0^m,\mb{X}'')\}\).  Then it is immediate that \(\mc{H}''(\mb{X}'')\) is a \((t,\ell)\)-exact hyperplane cover for \(S\subsetneq\{0,1\}^{n-m}\).  This implies \(\EHC_n^{(t,\ell)}(\{0,1\}^m\times S)\ge\EHC_{n-m}^{(t,\ell)}(S)\).
	
	Conversely, let \(\mc{H}(\mb{X}'')=\{h_1(\mb{X}''),\ldots,h_q(\mb{X}'')\}\) be a \((t,\ell)\)-exact hyperplane cover for \(S\subsetneq\{0,1\}^{n-m}\) with \(q=|\mc{H}(\mb{X}'')|=\EHC_{n-m}^{(t,\ell)}(S)\).  Then again, it is immediate that \(\ol{\mc{H}}(\mb{X}',\mb{X}'')\coloneqq\mc{H}(\mb{X}'')\) is a \((t,\ell)\)-exact hyperplane cover for \(\{0,1\}^m\times S\).  This implies \(\EHC_n^{(t,\ell)}(\{0,1\}^m\times S)\le\EHC_{n-m}^{(t,\ell)}(S)\).  Thus, we have proved the first identity.
	
	The second identity then follows immediately from Theorem~\ref{thm:multiplicity-symmetric}.
\end{proof}

\subsection{Our results: higher multiplicity polynomial covers}\label{subsec:results-polynomial}

Let us now look at a few instances where we can solve the polynomial covering problem in broader generality, but not the hyperplane covering problem.  In fact, in this extended setting, we will also need some \emph{nondegeneracy conditions} to obtain a clean combinatorial characterization.

Consider the hypercube \(\{0,1\}^N=\{0,1\}^{n_1}\times\cdots\times\{0,1\}^{n_k}\).  Recall that we will now work with the indeterminates \(\mb{X}=(\mb{X}_1,\ldots,\mb{X}_k)\), where \(\mb{X}_j=(X_{j,1},\ldots,X_{j,n_j})\) are the indeterminates for the \(j\)-th block.  Let \(t\ge1,\,\ell\in[0,t-1]\), and consider any subset \(S\subseteq\{0,1\}^N\).  We define
\begin{itemize}
	\item  a \tsf{\((t,\ell)\)-block exact hyperplane cover} for \(S\) to be a \((t,\ell)\)-exact hyperplane cover \(\mc{H}(\mb{X})\) (in \(\mb{R}^N\)) for \(S\) such that
	\[
	|\mc{H}(a,\mb{X}_j)|=|\mc{H}(\mb{X})|,\quad\tx{for every }a\in\{0,1\}^{n_1}\times\cdots\times\{0,1\}^{n_{j-1}}\times\{0,1\}^{n_{j+1}}\times\cdots\times\{0,1\}^{n_k},\,j\in[k].
	\]
	
	\item  a \tsf{\((t,\ell)\)-block exact polynomial cover} for \(S\) to be a nonzero polynomial \(P(\mb{X})\in\mb{R}[\mb{X}]\) such that
	\begin{enumerate}[(a)]
		\item  the polynomial \(P(\mb{X})\) vanishes at each point in \(S\) with multiplicity at least \(t\),
		\item  for each \(j\in[k]\), and every point \((a,\wt{a})\in\{0,1\}^N\setminus S\) with \(a\in\{0,1\}^{n_1}\times\cdots\times\{0,1\}^{n_{j-1}}\times\{0,1\}^{n_{j+1}}\times\cdots\times\{0,1\}^{n_k},\,\wt{a}\in\{0,1\}^{n_j}\), the polynomial \(P(a,\mb{X}_j)\) vanishes at \(\wt{a}\) with multiplicity exactly \(\ell\).
	\end{enumerate}
\end{itemize}
Let \(\bEHC_{(n_1,\ldots,n_k)}^{(t,\ell)}(S)\) denote the minimum size of a \((t,\ell)\)-block exact hyperplane cover for \(S\), and let \(\bEPC_{(n_1,\ldots,n_k)}^{(t,\ell)}(S)\) denote the minimum degree of a \((t,\ell)\)-block exact polynomial cover for \(S\).  It is obvious from the definitions that, in general, we have \(\bEHC_{(n_1,\ldots,n_k)}^{(t,\ell)}(S)\ge\bEPC_{(n_1,\ldots,n_k)}^{(t,\ell)}(S)\).  For completeness, we give a quick proof in Appendix~\ref{app:bEHC-bEPC}.  Further, it is trivial from the definitions that \(\bEHC_{(n_1,\ldots,n_k)}^{(t,\ell)}(S)\ge\EHC_N^{(t,\ell)}(S)\) and \(\bEPC_{(n_1,\ldots,n_k)}^{(t,\ell)}(S)\ge\EPC_N^{(t,\ell)}(S)\).  A blockwise variant of Question~\ref{ques:main} that we will consider is the following.
\begin{question}\label{ques:block-main}
	Given a proper subset \(S\subsetneq\{0,1\}^N\) and integers \(t\ge1,\,\ell\in[0,t-1]\), under what conditions can we say that \(\bEHC_{(n_1,\ldots,n_k)}^{(t,\ell)}(S)=\bEPC_{(n_1,\ldots,n_k)}^{(t,\ell)}(S)\)?
\end{question}
\noindent Unfortunately, we are unable to answer Question~\ref{ques:block-main} in the generality that we consider; in fact, we suspect that the answer could be negative.  Instead, we can solve simply the blockwise polynomial covering problem.
\begin{question}\label{ques:block-poly}
	Given a proper subset \(S\subsetneq\{0,1\}^N\) and integers \(t\ge1,\,\ell\in[0,t-1]\), under what conditions can we (combinatorially) characterize \(\bEPC_{(n_1,\ldots,n_k)}^{(t,\ell)}(S)\)?
\end{question}

\subsubsection{Covering \emph{pseudo downward closed} (PDC) \(k\)-wise symmetric sets}\label{subsubsec:covering-PDC-k-symmetric}

Our proof technique extends further to a more general class of \(k\)-wise symmetric sets to give a characterization for the blockwise polynomial covering problem, that is, we answer Question~\ref{ques:block-poly}.  In fact, the tight polynomial construction for this characterization hints that in this generality, the answers to Question~\ref{ques:main} and Question~\ref{ques:block-main} could be negative.

Consider the two obvious total orders \(\le\) and \(\le'\) on \(\mb{N}\) defined by
\[
0<1<2<3<\cdots\quad\tx{and}\quad0>'1>'2>'3>'\cdots
\]
Let \(\ms{T}=\{\le,\le'\}\).  For any \(S\subseteq\{0,1\}^N\) and \(j\in[k]\), let \(S_j\subseteq\{0,1\}^{n_j}\) denote the \tsf{projection} of \(S\) onto the \(j\)-th block.  Consider any \(k\)-wise symmetric set \(S\subseteq\{0,1\}^N\).  It is immediate that each \(S_j\) is symmetric, \(S_1\times\cdots\times S_k\) is a \(k\)-wise grid, and \(S\subseteq S_1\times\cdots\times S_k\).  Further, denote
\[
W_{(n_1,\ldots,n_k)}(S)=\{(|x_1|,\ldots,|x_k|):(x_1,\ldots,x_k)\in S\}.
\]
Then clearly, \(W_{(n_1,\ldots,n_k)}(S)\subseteq W_{n_1}(S_1)\times\cdots\times W_{n_k}(S_k)\).  For each \(j\in[k]\), we consider an arbitrarily chosen total order \(\le_j\,\in\ms{T}\) on \(W_{n_j}(S_j)\), say denoted by \(W_{n_j}(S_j)=\{w_{j,0}<_j\cdots<_jw_{j,q_j}\}\), and further for each \(z_j\in[0,q_j]\), define the symmetric set \([S]_{j,z_j}\subseteq\{0,1\}^{n_j}\) by \(W_n([S]_{j,z_j})=\{w_{j,0}<_j\cdots<_jw_{j,z_j}\}\).

We define a \(k\)-symmetric set \(S\subseteq\{0,1\}^N\) to be \tsf{pseudo downward closed (PDC)} if for every \((w_{1,z_1},\ldots,w_{k,z_k})\in W_{(n_1,\ldots,n_k)}(S)\) we have \(W_{n_1}([S]_{1,z_1})\times\cdots\times W_{n_k}([S]_{k,z_k})\subseteq W_{(n_1,\ldots,n_k)}(S)\), that is, \(W_{(n_1,\ldots,n_k)}(S)\) is \emph{downward closed}\footnote{For any poset \((P,\le)\), a subset \(D\subseteq P\) is downward closed if for any \(x\in D\) we have \(y\in D\) for all \(y\in P,\,y\le x\).} in \(W_{n_1}(S_1)\times\cdots\times W_{n_k}(S_k)\) under the partial order induced by \(\le_1,\ldots,\le_k\).    Further, let
\[
\mc{N}(S)=\{(z_1,\ldots,z_k)\in\mb{N}^k:(w_{1,z_1},\ldots,w_{k,z_k})\in W_{(n_1,\ldots,n_k)}(S)\}.
\]
It is clear that \(\mc{N}(S)\) is downward closed in \(\mb{N}^k\).  Also let \(\outext(S)\) denote the set of all minimal elements of the complement set \(\mb{N}^k\setminus \mc{N}(S)\) with respect to the natural partial order on \(\mb{N}^k\).

It is quite easy to check that the complement of a PDC \(k\)-symmetric set is again PDC \(k\)-symmetric.  We defer the proof to the Appendix.  Our third main result generalizes Theorem~\ref{thm:multiplicity-block-symmetric}, but solves only the block polynomial covering problem, that is, answers Question~\ref{ques:block-poly}.  Note that in this generality, the combinatorial characterization that we have is nicer to describe in terms of complements.\footnote{This is why, for consistency, we have retained the description in terms of complements throughout this work.}
\begin{theorem}\label{thm:EPC-PDC}
	For any nonempty PDC \(k\)-wise symmetric set \(S\subseteq\{0,1\}^N\) and \(t\ge1\), we have
	\[
	\bEPC_{(n_1,\ldots,n_k)}^{(t,t-1)}(\{0,1\}^N\setminus S)=\max_{(z_1,\ldots,z_k)\in\outext(S)}\bigg\{\sum_{j\in[k]:z_j\ge1} \ol{\Lambda}_{n_j}([S]_{j,z_j-1})\bigg\}+2t-2.
	\]
\end{theorem}
\begin{remark}\label{rem:PDC-equivalent}
	The proof of Theorem~\ref{thm:EPC-PDC} will also show that for any nonempty PDC \(k\)-wise symmetric set \(S\subseteq\{0,1\}^N\) and \(t\ge1\), we have
	\begin{align*}
		\bEPC_{(n_1,\ldots,n_k)}^{(t,t-1)}(\{0,1\}^N\setminus S)&=\bEPC_{(n_1,\ldots,n_k)}^{(1,0)}(\{0,1\}^N\setminus S)+2t-2\\
		&=\max_{(z_1,\ldots,z_k)\in\outext(S)}\bigg\{\sum_{i\in[k]:z_i\ge1}\ol{\Lambda}_{n_j}([S]_{j,z_j-1})\bigg\}+2t-2.
	\end{align*}
\end{remark}
\noindent A construction that implies the equality in Theorem~\ref{thm:EPC-PDC} can be adapted from Example~\ref{ex:multi-block-symm} as follows.
\begin{example}\label{ex:EPC-PDC}
	For any fundamental family of hyperplanes \(\mc{H}(\mb{X})=\{H_1(\mb{X}),\ldots,H_p(\mb{X})\}\) defined in Section~\ref{subsubsec:hyperplane-construction}, let us abuse notation and also denote the corresponding product polynomial by \(\mc{H}(\mb{X})=H_1(\mb{X})\cdots H_p(\mb{X})\).  Let \(S\subseteq\{0,1\}^N\) be a nonempty PDC \(k\)-wise symmetric set, and \(t\ge1\).  Assuming notations as in Example~\ref{ex:multi-block-symm}, for each \((z_1,\ldots,z_k)\in\outext(S)\), define
	\[
	\mc{H}_{S,(z_1,\ldots,z_k)}(\mb{X})=\prod_{j\in[k]:z_j\ge1}\Big(\mc{H}^*_{\ol{\mu}_{n_j}([S]_{j,z_j-1})}(\mb{X}_j)\cdot\mc{H}'_{W_{n_j}(\{0,1\}^{n_j}\setminus[S]_{j,z_j-1})\setminus W_{n_j,\ol{\mu}_{n_j}([S]_{j,z_j-1})}}(\mb{X}_j)\Big).
	\]
	Now consider a subfield of \(\mb{R}\) defined by \(\wh{\mb{Q}}=\mb{Q}\big(\mc{H}_{S,(z_1,\ldots,z_k)}(b):b\in\{0,1\}^N,\,(z_1,\ldots,z_k)\in\outext(S)\big)\).  It follows that \(\mb{R}\) is an infinite dimensional \(\wh{\mb{Q}}\)-vector space.  Choose any \(\wh{\mb{Q}}\)-linearly independent subset \(\{\lambda_{S,(z_1,\ldots,z_k)}:(z_1,\ldots,z_k)\in\outext(S)\}\subseteq\mb{R}\).  Then the polynomial
	\[
	\bigg(\sum_{(z_1,\ldots,z_k)\in\outext(S)}\lambda_{S,(z_1,\ldots,z_k)}\mc{H}_{S,(z_1,\ldots,z_k)}(\mb{X})\bigg)\cdot\mc{H}^{\circ(t-1)}(\mb{X}_1)
	\]
	witnesses the equality in Theorem~\ref{thm:EPC-PDC}.
\end{example}

\subsubsection*{Covering \(k\)-wise grids}\label{subsubsec:covering-k-grids}

Note that both \(k\)-wise grids and their complements are special cases of PDC \(k\)-wise symmetric sets.  So our first two main results (Theorem~\ref{thm:multiplicity-symmetric}, and Theorem~\ref{thm:multiplicity-block-symmetric} via Corollary~\ref{cor:grid-equivalence}) are, in fact, corollaries of our third main result (Theorem~\ref{thm:EPC-PDC}).  Further, the tight example of a polynomial cover mentioned in Example~\ref{ex:EPC-PDC} specializes to the tight examples of hyperplane covers mentioned in Example~\ref{ex:multi-symm} and Example~\ref{ex:multi-block-symm}.  

Theorem~\ref{thm:multiplicity-block-symmetric} characterizes the hyperplane and polynomial covering problems for complements of \(k\)-wise grids.  In this case, appealing to Theorem~\ref{thm:EPC-PDC}, we get the following, we see that the blockwise variants of our covering problems are equivalent to the usual \emph{non-blockwise} covering problems.
\begin{corollary}\label{cor:grid-equivalence}
	For any nonempty \(k\)-wise grid \(S=S_1\times\cdots\times S_k\subseteq\{0,1\}^N\) and \(t\ge1\), we have
	\begin{align*}
	\bEPC_{(n_1,\ldots,n_k)}^{(t,t-1)}(\{0,1\}^N\setminus S)&=\EPC_N^{(t,t-1)}(\{0,1\}^N\setminus S)\\
	&=\sum_{j=1}^k\ol{\Lambda}_{n_j}(S_j)+2t-2\\
	&=\EHC_N^{(t,t-1)}(\{0,1\}^N\setminus S)=\bEHC_{(n_1,\ldots,n_k)}^{(t,t-1)}(\{0,1\}^N\setminus S).
	\end{align*}
\end{corollary}
\noindent  Further, when it comes to covering \(k\)-wise grids (and not their complements), we get the following as a corollary of Theorem~\ref{thm:EPC-PDC}.
\begin{corollary}\label{cor:covering-grid}
	For any \(k\)-wise grid \(S=S_1\times\cdots\times S_k\subsetneq\{0,1\}^N\) and \(t\ge1\), we have
	\[
	\bEPC_{(n_1,\ldots,n_k)}^{(t,t-1)}(S)=\max\big\{\Lambda_{n_j}(S_j):j\in[k]\big\}+2t-2.
	\]
\end{corollary}
\noindent A construction that implies the equality in Corollary~\ref{cor:covering-grid} is a special case of Example~\ref{ex:EPC-PDC}.
\begin{example}\label{ex:covering-grid}
	Let \(S=S_1\times\cdots\times S_k\subsetneq\{0,1\}^N\) be a nonempty \(k\)-wise grid, and \(t\ge1\).  For each \(j\in[k]\), define
	\[
	\mc{H}_{S_j}(\mb{X}_j)=\mc{H}^*_{\ol{\mu}_{n_j}(S_j)}(\mb{X}_j)\cdot\mc{H}'_{W_{n_j}(\{0,1\}^{n_j}\setminus S_j)\setminus W_{n_j,\ol{\mu}_{n_j}(S_j)}}(\mb{X}_j).
	\]
	Now consider a subfield of \(\mb{R}\) defined by \(\wh{\mb{Q}}=\mb{Q}\big(\mc{H}_{S_j}(b):b\in\{0,1\}^{n_j},\,j\in[k]\big)\).  It follows that \(\mb{R}\) is an infinite dimensional \(\wh{\mb{Q}}\)-vector space.  Choose any \(\wh{\mb{Q}}\)-linearly independent subset \(\{\lambda_1,\ldots,\lambda_k\}\subseteq\mb{R}\).  Then the polynomial
	\[
	\bigg(\sum_{j=1}^k\lambda_j\mc{H}_{S_j}(\mb{X}_j)\bigg)\cdot\mc{H}^{\circ(t-1)}(\mb{X}_1)
	\]
	witnesses the equality in Corollary~\ref{cor:covering-grid}.
\end{example}

\subsubsection{Partial results on other multiplicity polynomial covers}\label{subsubsec:partial-covers}

Let us now mention a couple of results on \((t,0)\)-exact polynomial covers.  The first result concerns the polynomial covering problem for the \emph{Hamming ball}, which is a symmetric set defined by a set of weights of the form \([0,w]\).
\begin{proposition}\label{pro:EPC-hamming}
	For \(w\in[0,n-1]\), let \(S\subsetneq\{0,1\}^n\) be the symmetric set defined by \(W_n(S)=[0,w-1]\).  Then for any \(t\in\big[2,\big\lfloor\frac{n+3}{2}\big\rfloor\big]\), we have
	\[
	\EPC_n^{(t,0)}(S)=w+2t-3.
	\]
	Further, the answer to Question~\ref{ques:main} is negative, in general.
\end{proposition}

The second result concerns the polynomial covering problem for a single layer.  Surprisingly, in this case, our proof employs basic analytic facts about coordinate transformations of polynomials, but we do not know of a proof via the polynomial method.
\begin{proposition}\label{pro:EPC-layer}
	For any layer \(S\subsetneq\{0,1\}^n\) with \(W_n(S)=\{w\}\), and \(t\ge1\), we have
	\[
	\EPC_n^{(t,0)}(S)=t.
	\]
\end{proposition}

\subsubsection{Cool-down: Index complexity of PDC \(k\)-wise symmetric sets}\label{subsubsec:cooldown}

We conclude this by noting that the index complexity, which is a weaker notion for the blockwise covering problems that we consider, can be characterized to a good extent, even in the generality of PDC \(k\)-wise symmetric sets.  Note that for symmetric sets \(S,S'\subseteq\{0,1\}^n\) with \(S'\subseteq S\), if \(J_{n,a,b}=\outint(S)\), then \(J_{n,a,b}=\outint(S')\) if and only if \(\{a,b\}\subseteq W_n(S')\).  This turns out to be an important structural feature that we will work with.

Assume the block decomposition of the hypercube \(\{0,1\}^N=\{0,1\}^{n_1}\times\cdots\times\{0,1\}^{n_k}\).  Now let \(S\subseteq\{0,1\}^N\) be a nonempty PDC \(k\)-wise symmetric set.  Further, for each \(j\in[k]\), consider \(S_j\subseteq\{0,1\}^{n_j}\) (the \(j\)-th projection of \(S\)) and let \(J_{n_j,a_j,b_j}=\outint(S_j)\).  We define \(S\) to be \tsf{outer intact} if for every \((z_1,\ldots,z_k)\in\innext(S)\) and \(j\in[k]\), we have \(J_{n_j,a_j,b_j}=\outint([S]_{j,z_j})\).  Equivalently, \(S\) is outer intact if and only if
\[
\{a_j,b_j\}\subseteq W_{n_j}([S]_{j,z_j})\tx{ for each }j\in[k],\quad\tx{for every }(z_1,\ldots,z_k)\in\innext(S).
\]

\begin{proposition}\label{pro:PDC-index-complexity}
	For any nonempty outer intact PDC \(k\)-wise symmetric set \(S\subseteq\{0,1\}^N\), we have
	\[
	r_N(S)=\sum_{j=1}^kr_{n_j}(S_j)=\sum_{j=1}^k\out_{n_j}(S_j).
	\]
\end{proposition}
\noindent  An important special case of Proposition~\ref{pro:PDC-index-complexity} is for a \(k\)-wise layer, which is trivially outer intact PDC.  As an immediate corollary of Proposition~\ref{pro:PDC-index-complexity} and Proposition~\ref{pro:layer-index-complexity}, we get the following.
\begin{corollary}\label{cor:k-layer-index-complexity}
	For any \(k\)-wise layer \(S=S_1\times\cdots\times S_k\subseteq\{0,1\}^N\) with \(W_{n_j}(S_j)=\{w_j\},\,j\in[k]\), we have
	\[
	r_N(S)=\sum_{j=1}^k\min\{w_j,n_j-w_j\}.
	\]
\end{corollary}

Proposition~\ref{pro:PDC-index-complexity} shows that the index complexity is sensitive only to the blockwise projections, but Theorem~\ref{thm:EPC-PDC} (for any PDC \(k\)-wise symmetric set) shows that the characterization of the polynomial covering problem is more sensitive to the \emph{specific PDC structure}.  This adds to our observation that our polynomial method argument is stronger than simply giving a lower bound in terms of index complexity.

\subsection{Related work}\label{subsec:related}

In addition to the works that motivated our results, there is a plethora of literature on hyperplane covering problems and related questions, over both the reals as well as finite fields.  Even more, the polynomial method itself has been subject to intense investigation in the last few decades.  We mention here a sample from this vast literature that we believe is most relevant to our present work.

\subsubsection*{Hyperplane covering problems}
\begin{itemize}
	\item  Alon, Bergmann, Coppersmith, and Odlyzko studied a \emph{balancing problem} for sets of binary vectors, which admits a simple reformulation as a hyperplane covering problem.  An extension of this problem to higher order complex roots of unity, which takes the form of a polynomial covering problem, was studied by Heged\H{u}s~\cite{hegedus-2010-balancing}.
	
	\item  K\'os, M\'esz\'aros, and R\'onyai~\cite{kos2012alon} extended the result of Alon and F\"uredi~\cite{alon-furedi} to the case where the vanishing constraints at every point of the hypercube have multiplicities depending on the individual coordinates of the point.  The question in~\cite{alon-furedi} itself was extracted by B\'ar\'any from the work of Komj\'ath~\cite{komjath1994partitions}.
	
	\item  Linial and Radhakrishnan~\cite{linial2005essential} considered the notion of an \emph{essential hyperplane cover} for the hypercube, which is a minimal family of hyperplanes that are sufficiently \emph{oblique}, and such that every coordinate influences at least one hyperplane.  They gave an upper bound of \(\lfloor n/2\rfloor+1\) and a lower bound of \(\Omega(\sqrt{n})\).  Saxton~\cite{saxton2013essential} gave a tight bound of \(n+1\) in the special case wherein the coefficients of all the variables in the affine linear polynomials representing the hyperplanes are restricted to be nonnegative.  Recent breakthroughs by Yehuda and Yehudayoff~\cite{yehuda2021lower}, and Araujo, Balogh, and Mattos~\cite{balogh2022essential} have improved the lower bound to \(n^{5/9-o(1)}\).
	
	\item  Several extensions and variants of covering problems over finite fields have appeared in the language of hyperplanes as well as in the dual language of \emph{blocking sets}, and the proof techniques in most of these works involve the polynomial method -- Jamison~\cite{JAMISON1977253}, Brouwer~\cite{BROUWER1978251}, Ball~\cite{BALL2000441}, Zanella~\cite{ZANELLA2002381}, Ball and Serra~\cite{ball2009punctured}, Blokhuis~\cite{blokhuis2010covering}, and Bishnoi, Boyadzhiyska, Das and M\'esz\'aros~\cite{bishnoi2021subspace}, to name a few.  
\end{itemize}

\subsubsection*{The polynomial method}
\begin{itemize}
	\item  One of the simplest ways to formally encapsulate the polynomial method is via a classical algebraic object called the \emph{finite-degree Zariski closure}.  It was defined by Nie and Wang~\cite{nie2015hilbert} in the context of combinatorial geometry over finite fields, who studied bounds on its size for arbitrary subsets of the hypercube.  However, it had been studied implicitly even earlier by, for instance, Wei~\cite{wei-1991-GHM}, Heijnen and Pellikaan~\cite{heijnen-pellikaan-1998-GHM-Reed-Muller}, Keevash and Sudakov~\cite{keevash-sudakov-2005-min-rank-inclusion}, and Ben-Eliezer, Hod, and Lovett~\cite{ben-eliezer-hod-lovett-2012-low-degree-polys}.  Attempts to characterize the finite-degree Zariski closures of symmetric sets of the hypercube were done in the works of Heged\H{u}s~\cite{hegedus-2010-balancing,hegedus-2021-L-balancing}, the fourth author~\cite{venkitesh-2022-covering}, as well as Srinivasan and the fourth author~\cite{srinivasan-venkitesh-2021-mod-p} (and also implicitly in Bernasconi and Egidi~\cite{bernasconi-egidi-1999-hilbert-symmetric}).
	
	\item  A stronger notion than finite-degree Zariski closure is another algebraic object called the \emph{affine Hilbert function}.  The affine Hilbert functions of all layers of the hypercube over all fields were determined by Wilson~\cite{wilson-1990-diagonal-incidence}.  Further, Bernasconi and Egidi~\cite{bernasconi-egidi-1999-hilbert-symmetric} determined the affine Hilbert functions of all symmetric sets of the hypercube over the reals.  This was extended to the setting of larger grids by the fourth author~\cite{venkitesh-2021-Hilbert}.
	
	\item  An even stronger notion than affine Hilbert functions is yet another algebraic object called the \emph{Gr\"obner basis}, along with the associated collection of \emph{standard monomials}.  Anstee, R\'onyai, and Sali~\cite{anstee-ronyai-sali-2002-shattering}, and Friedl and R\'onyai~\cite{friedl-ronyai-2003-order-shattering} studied the standard monomials for any subset of the hypercube in terms of a combinatorial phenomenon called \emph{order shattering}.  Felszeghy, R\'ath, and R\'onyai~\cite{felszeghy-rath-ronyai-2006-lex} characterized the standard monomials of all symmetric sets of the hypercube via a \emph{lex game}.  Heged\H{u}s and R\'onyai~\cite{hegedus-ronyai-2003-grobner-complete-uniform,hegedus-ronyai-2018-linear-sperner}, and Felszeghy, Heged\H{u}s, and R\'onyai~\cite{felszeghy-hegedus-ronyai-2009-complete-wide} characterized the Gr\'obner basis for special cases of symmetric sets of the hypercube.
\end{itemize}

\subsection*{Organization of the paper}\label{subsec:organization}

In Section~\ref{sec:prelims}, we begin by covering some preliminaries, as well as setup some terminologies and notations.  In Section~\ref{sec:index-complexity-PDC}, we will obtain characterizations of index complexity for the symmetry preserving subsets that we are interested in.  This covers our warmup results (SectioN~\ref{subsubsec:warmup}) and our cooldown results (Section~\ref{subsubsec:cooldown}).  In Section~\ref{sec:covering-PDC}, we will prove our third main result -- a characterization for the blockwise polynomial covering problem (Theorem~\ref{thm:EPC-PDC} in Section~\ref{subsubsec:covering-PDC-k-symmetric}).  We will also note that our first main result (Theorem~\ref{thm:multiplicity-symmetric} in Section~\ref{subsubsec:covering-symmetric}), our second main result (Theorem~\ref{thm:multiplicity-block-symmetric} in Section~\ref{subsubsec:covering-k-symmetric}), as well as all other results in Section~\ref{subsubsec:covering-symmetric}, Section~\ref{subsubsec:covering-k-symmetric}, and Section~\ref{subsubsec:covering-PDC-k-symmetric} are corollaries of Theorem~\ref{thm:EPC-PDC}.  In Section~\ref{sec:partial-results}, we will prove our partial results on other higher multiplicity polynomial covers (Section~\ref{subsubsec:partial-covers}).  Finally, in Section~\ref{sec:conclusion}, we conclude with a discussion on some open questions.

\section{Preliminaries}\label{sec:prelims}

In this section, we will refresh some essential preliminary notions, as well as setup terminologies and notations.

\subsection{Posets}\label{subsec:prelims-posets}

Let \((P,\le)\) be a poset, that is, \(\le\) is a partial order on a nonempty set \(P\).  For a subset \(S\subseteq P\), we denote \(\min_\le(S)\) to be the set of all \tsf{minimal elements} of \(S\), and \(\max_\le(S)\) to be the set of all \tsf{maximal elements}, that is,
\begin{align*}
\textstyle\min_\le(S)&=\{a\in S:(b\in S,\,b\le a)\implies b=a\},\\
\textstyle\max_\le(S)&=\{a\in S:(b\in S,\,b\ge a)\implies b=a\}.
\end{align*}
Further, we define the sets of \tsf{outer extremal elements} and \tsf{inner extremal elements} of \(S\), respectively, by
\begin{align*}
\textstyle\outext_\le(S)&=\textstyle\min_\le(P\setminus S),\\
\tx{and}\quad\textstyle\innext_\le(S)&=\textstyle\max_\le(S).
\end{align*}
A subset \(S\subseteq P\) is defined to be \tsf{downward closed} if
\[
a\in S,\,b\in P,\,b\le a\quad\implies\quad b\in P.
\]

For two posets \((P_1,\le_1)\) and \((P_2,\le_2)\), the \tsf{product poset} is the poset \((P_1\times P_2,\le)\), where \(\le\) is defined by
\[
(a_1,a_2)\le(b_1,b_2)\tx{ if and only if }a_1\le_1 b_1\tx{ and }a_2\le_2b_2.
\]
We also say \(\le\) is the \tsf{induced order} on \(P_1\times P_2\).

If we consider the obvious total order \(\le\) on \(\mb{N}\) given by \(0<1<2<3<\cdots\), then the induced order on \(\mb{N}^k\) is called the \tsf{natural order} on \(\mb{N}^k\).

\subsection{Symmetry preserving subsets of the hypercube}\label{subsec:prelims-symmetry}

We are interested in hyperplane and polynomial covering problems for some structured subsets of the hypercube \(\{0,1\}^n\), where the \emph{structures} that we are concerned with are specified by invariance under the action of some subgroups of the symmetric group \(\mf{S}_n\).

\paragraph*{Symmetric sets.}  Let \(S\subseteq\{0,1\}^n\).  We say \(S\) is \tsf{symmetric} if
\[
(x_1,\ldots,x_n)\in S\tx{ and }\sigma\in\mf{S}_n\quad\implies\quad(x_{\sigma(1)},\ldots,x_{\sigma(n)})\in S.
\]
It follows immediately that \(S\) is symmetric if and only if
\[
x\in S,\,y\in\{0,1\}^n,\tx{ and }|y|=|x|\quad\implies\quad y\in S.
\]
In this case, we denote \(W_n(S)=\{|x|:x\in S\}\subseteq[0,n]\).  So the symmetric set \(S\) is completely determined by \(W_n(S)\).  If \(|W_n(S)|=1\), then we say \(S\) is a \tsf{layer}.  It is immediate that a subset of the hypercube is symmetric if and only if it is a union of some collection of layers.

\paragraph*{Two combinatorial measures.}  For any \(x\in\{0,1\}^n\), the \tsf{Hamming weight} of \(x\) is defined by \(|x|=\{i\in[n]:x_i=1\}\).  For any subset of coordinates \(I\subseteq[n]\), we denote \(x_I=(x_i:i\in I)\in\{0,1\}^I\).  We require a simple combinatorial measure defined in~\cite{ghosh-kayal-nandi-2023-covering}.  For a subset \(S\subseteq\{0,1\}^n\), the \tsf{index complexity} is defined by
\[
r_n(S)=\min\{|I|:I\subseteq[n],\,\tx{there exists \(a\in S\) such that }b_I\ne a_I\tx{ for all }b\in S,\,b\ne a\}.
\]
So \(r_n(S)\) is the minimum number of coordinates required to \emph{separate some element in \(S\) from all other elements in \(S\)}.

An important symmetric set that we will need consists of elements with Hamming weights in an \emph{initial interval} of weights or a \emph{final interval} of weights.  For \(i\in[0,n]\), define \(W_{n,i}=[0,i-1]\cup[n-i+1,n]\), and the symmetric set \(T_{n,i}\subseteq\{0,1\}^n\) by \(W_n(T_{n,i})=W_{n,i}\).  We also require another combinatorial measure, that is specific to symmetric sets, defined in~\cite{venkitesh-2022-covering}.  For any symmetric set \(S\subseteq\{0,1\}^n\), define
\begin{align*}
\mu_n(S)&=\max\{i\in[0,\lceil n/2\rceil]:W_{n,i}\subseteq W_n(S)\},\\
\tx{and}\quad\Lambda_n(S)&=|W_n(S)|-\mu_n(S).
\end{align*}
Further, we denote \(\ol{\mu}_n(S)\coloneqq\mu_n(\{0,1\}^n\setminus S)\) and \(\ol{\Lambda}_n(S)\coloneqq\Lambda_n(\{0,1\}^n\setminus S)\).  We will also need a simple fact about the invariance of the above two combinatorial measures under \emph{complementation of coordinates}.  It follows straightforwardly from the definitions, and we give a proof in Appendix~\ref{app:complement}.
\begin{fact}\label{fact:transform}
	Let \(S\subseteq\{0,1\}^n\) be a symmetric set, and \(\wt{S}\) be the image of \(S\) under the coordinate transformation \((X_1,\ldots,X_n)\mapsto(1-X_1,\ldots,1-X_n)\).
	\begin{enumerate}[(a)]
		\item  If \(S\ne\{0,1\}^n\), then \(\Lambda_n(\wt{S})=\Lambda_n(S)\).
		\item  If \(S\ne\emptyset\), then \(r_n(\wt{S})=r_n(S)\).
	\end{enumerate}
\end{fact}

\paragraph*{\emph{Blockwise} symmetric sets.}  Now fix a \emph{block decomposition} of the hypercube as \(\{0,1\}^N=\{0,1\}^{n_1}\times\cdots\times\{0,1\}^{n_k}\).  Let \(S\subseteq\{0,1\}^N\).  We say \(S\) is \tsf{\(k\)-wise symmetric} if
\begin{align*}
	&\big((x_{1,1},\ldots,x_{1,n_1}),\ldots,(x_{k,1},\ldots,x_{k,n_k})\big)\in S\tx{ and }(\sigma_1,\ldots,\sigma_k)\in\mf{S}_{n_1}\times\cdots\times\mf{S}_{n_k}\\
	\implies\quad&\big((x_{1,\sigma_1(1)},\ldots,x_{1,\sigma_1(n_1)}),\ldots,(x_{k,\sigma_k(1)},\ldots,x_{k,\sigma_k(n_k)})\big)\in S.
\end{align*}
It follows immediately that \(S\) is \(k\)-wise symmetric if and only if
\begin{align*}
	&(x_1,\ldots,x_k)\in S,\,(y_1,\ldots,y_k)\in\{0,1\}^N,\tx{ and }|y_i|=x_i\tx{ for all }i\in[k]\\
	\implies\quad&(y_1,\ldots,y_k)\in S.
\end{align*}
In this case, we denote \(W_{(n_1,\ldots,n_k)}(S)=\{(|x_1|,\ldots,|x_k|):(x_1,\ldots,x_k)\in S\}\subseteq[0,n_1]\times\cdots\times[0,n_k]\).  So the \(k\)-wise symmetric set \(S\) is completely determined by \(W_{(n_1,\ldots,n_k)}(S)\).  For each \(j\in[k]\), let \(S_j\subseteq\{0,1\}^{n_j}\) denote the \tsf{\(j\)-th projection} of \(S\), that is, \(S_j=\{x_j\in\{0,1\}^{n_j}:(x_1,\ldots,x_k)\in S\}\).  So we clearly have \(W_{(n_1,\ldots,n_k)}(S)\subseteq W_{n_1}(S_1)\times\cdots W_{n_k}(S_k)\).  We say \(S\) is a \tsf{\(k\)-wise grid} if \(W_{(n_1,\ldots,n_k)}(S)=W_{n_1}(S_1)\times\cdots\times W_{n_k}(S_k)\).  We say a \(S\) is a \tsf{\(k\)-wise layer} if \(|W_{(n_1,\ldots,n_k)}(S)|=1\), or equivalently, each \(S_j\) is a layer.  It is immediate that a subset of a hypercube is \(k\)-wise symmetric if and only if it is a union of some collection of \(k\)-wise layers.

Consider the two obvious total orders \(\le\) and \(\le'\) on \(\mb{N}\) defined by
\[
0<1<2<3<\cdots\quad\tx{and}\quad0>'1>'2>'3>'\cdots
\]
Let \(\ms{T}=\{\le,\le'\}\).  Let \(S\subseteq\{0,1\}^N\) be \(k\)-wise symmetric.  Fix arbitrary total orders \(\le_j\in\ms{T}\) on \(W_{n_j}(S_j)\) for each \(j\in[k]\), and consider the induced partial order \(\preceq\) on \(W_{n_1}(S_1)\times\cdots\times W_{n_k}(S_k)\).  We define \(S\) to be \tsf{pseudo downward closed (PDC)} if \(W_{(n_1,\ldots,n_k)}(S)\) is downward closed in \(W_{n_1}(S_1)\times\cdots\times W_{n_k}(S_k)\).  Further, for all \(j\in[k]\), enumerate \(W_{n_j}(S_j)=\{w_{j,0}<_j\cdots<_jw_{j,q_j}\}\), and for each \(z_j\in[0,q_j]\), define the symmetric set \([S]_{j,z_j}\subseteq\{0,1\}^{n_j}\) by \(W_{n_j}([S]_{j,z_j})=\{w_{j,0}<_j\cdots<_jw_{j,z_j}\}\).  Then define
\[
\mc{N}(S)=\{(z_1,\ldots,z_k)\in\mb{N}^k:(w_{1,z_1},\ldots,w_{k,n_k})\in W_{(n_1,\ldots,n_k)}(S)\}.
\]
It is immediate that the following are both equivalent conditions to \(S\) being PDC.
\begin{itemize}
	\item  \(\mc{N}(S)\) is downward closed in \(\mb{N}^k\) with respect to the natural order (also denoted by \(\le\)).
	\item  \(W_{n_1}([S]_{1,z_1})\times\cdots\times W_{n_k}([S]_{k,z_k})\subseteq W_{(n_1,\ldots,n_k)}(S)\) for each \((z_1,\ldots,z_k)\in \mc{N}(S)\).
\end{itemize}
We will also need two simple indexing sets in our results.  We denote
\begin{align*}
\outext(S)&\coloneqq\outext_\le(\mc{N}(S))=\{(z_1,\ldots,z_k)\in\mb{N}^k:(w_{1,z_1},\ldots,w_{k,z_k})\in\outext_{\preceq}(W_{(n_1,\ldots,n_k)}(S))\},\\
\innext(S)&\coloneqq\innext_\le(\mc{N}(S))=\{(z_1,\ldots,z_k)\in\mb{N}^k:(w_{1,z_1},\ldots,w_{k,z_k})\in\innext_{\preceq}(W_{(n_1,\ldots,n_k)}(S))\}.
\end{align*}

\subsection{Polynomials, multiplicities, hyperplanes, and covers}\label{subsec:prelims-vanishing}

We will work with the polynomial ring \(\mb{R}[\mb{X}]\), where \(\mb{X}=(X_1,\ldots,X_n)\) are the indeterminates.  We are interested in higher order vanishing properties of polynomials.  Let \(P(\mb{X})\in\mb{R}[\mb{X}]\).  For any \(\alpha=(\alpha_1,\ldots,\alpha_n)\in\mb{N}^n\), denote \(|\alpha|=\alpha_1+\cdots+\alpha_n\).  We will denote the \tsf{\(\alpha\)-th order partial derivative} of \(P(\mb{X})\) by \(\partial^\alpha P(\mb{X})\), that is,
\[
\partial^\alpha P(\mb{X})\coloneqq\frac{\partial^{|\alpha|}P(\mb{X})}{\partial X_1^{\alpha_1}\cdots\partial X_n^{\alpha_n}}.
\]
For any \(t\ge0\) and \(a\in\mb{R}^n\), we define the \tsf{multiplicity of \(P(\mb{X})\) at \(a\)} as follows: we define \(\mult(P(\mb{X}),a)\ge t\) if \(\partial^\alpha P(a)=0\), for all \(\alpha\in\mb{N}^n,\,|\alpha|<t\).  Therefore, we get \(\mult(P(\mb{X}),a)=t\) if \(\mult(P(\mb{X}),a)\ge t\) and \(\partial^\alpha P(a)\ne0\) for some \(\alpha\in\mb{N}^n\) with \(|\alpha|=t\).

An \tsf{affine hyperplane} in \(\mb{R}^n\) is any set of the form \(K+v\), where \(K\subseteq\mb{R}^n\) is a vector subspace with \(\dim(K)=n-1\), and \(v\in\mb{R}^n\).  In the rest of the paper, we will drop the adjective `affine' and simply refer to these as hyperplanes.  A set \(H\subseteq\mb{R}^n\) is a hyperplane if and only if \(H=\mc{Z}(P)\coloneqq\{a\in\mb{R}^n:P(a)=0\}\) for some nonzero polynomial \(P(\mb{X})\in\mb{R}[\mb{X}]\) with \(\deg(P)=1\).  In fact, we will identify \(H\) with its defining affine linear polynomial, and denote \(P(\mb{X})\) by \(H(\mb{X})\).  So according to the context (which will be obvious), \(H(\mb{X})\) will either denote the hyperplane as a subset of \(\mb{R}^n\) or the defining affine linear polynomial.  Similarly, if \(\mc{H}(\mb{X})=\{H_1(\mb{X}),\ldots,H_k(\mb{X})\}\) is a family of hyperplanes, we may also abuse notation and denote the corresponding defining polynomial by \(\mc{H}(\mb{X})=H_1(\mb{X})\cdots H_k(\mb{X})\).  For our concern, the family \(\mc{H}(\mb{X})\) will be a multiset, and \(|\mc{H}(\mb{X})|\) will denote the multiset cardinality of the family, that is, the number of hyperplanes counted with repetition.

We are interested in covering\footnote{We say a polynomial \(P\) covers a point \(a\) if \(a\in\mc{Z}(P)\).  Similarly, we say \(P\) covers a point \(a\) with multiplicity at least \(t\) if \(a\in\mc{Z}^t(P)\).} subsets of the hypercube \(\{0,1\}^n\) by polynomials and families of hyperplanes.  Let \(S\subsetneq\{0,1\}^n\), and consider \emph{multiplicity parameters} \(t\ge1,\,\ell\in[0,t-1]\).  We define
\begin{itemize}
	\item  a nonzero polynomial \(P(\mb{X})\in\mb{R}[\mb{X}]\) to be a \tsf{\((t,\ell)\)-exact polynomial cover} for \(S\) if
	\begin{align*}
		&\mult(P(\mb{X}),a)\ge t\quad\tx{for all }a\in S,\\
		\tx{and}\quad&\mult(P(\mb{X}),b)=\ell\quad\tx{for all }b\in\{0,1\}^n\setminus S.
	\end{align*}
		
	\item  a finite multiset of hyperplanes \(\mc{H}(\mb{X})\) in \(\mb{R}^n\) to be a \tsf{\((t,\ell)\)-exact hyperplane cover} for \(S\) if
	\begin{align*}
		&|\{H(\mb{X})\in\mc{H}(\mb{X}):H(a)=0\}|\ge t\quad\tx{for all }a\in S,\\
		\tx{and}\quad&|\{H(\mb{X})\in\mc{H}(\mb{X}):H(b)=0\}|=\ell\quad\tx{for all }b\in\{0,1\}^n\setminus S.
	\end{align*}
	This implies that \(\mc{H}(\mb{X})\) is also a \((t,\ell)\)-exact polynomial cover for \(S\).
\end{itemize}
Let \(\EHC_n^{(t,\ell)}(S)\) denote the minimum size of a \((t,\ell)\)-exact hyperplane cover for \(S\), and let \(\EPC_n^{(t,\ell)}(S)\) denote the minimum degree of a \((t,\ell)\)-exact polynomial cover for \(S\).  The definitions immediately imply that \(\EHC_n^{(t,\ell)}(S)\ge\EPC_n^{(t,\ell)}(S)\).

\paragraph*{A covering result.}  In the results of Alon and F\"uredi~\cite{alon-furedi} (Theorem~\ref{thm:alon-furedi}), as well as Sauermann and Wigderson~\cite{sauermann-wigderson-2022-covering} (Theorem~\ref{thm:sauermann-wigderson}), there is nothing sacrosanct about the origin; one could instead choose to avoid any single point.  We will use these version of the results, and therefore state it here.
\begin{theorem}[{\cite{alon-furedi}}]\label{thm:alon-furedi-again}
	For any \(a\in\{0,1\}^n\), we have
	\[
	\EHC_n^{(1,0)}(\{0,1\}^n\setminus\{a\})=\EPC_n^{(1,0)}(\{0,1\}^n\setminus\{a\})=n.
	\]
\end{theorem}
\begin{theorem}[{\cite{sauermann-wigderson-2022-covering}}]\label{thm:sauermann-wigderson-again}
	For all \(t\ge1,\,\ell\in[0,t-1]\), and any \(a\in\{0,1\}^n\), we have
	\[
	\EPC_n^{(t,\ell)}(\{0,1\}^n\setminus\{a\})=\begin{cases}
		n+2t-2&\tx{if}\n\ell=t-1,\\
		n+2t-3&\tx{if}\n\ell<t-1\le\big\lfloor\frac{n+1}{2}\big\rfloor.
	\end{cases}
	\]
\end{theorem}

\paragraph*{Nondegenerate polynomial and hyperplane covers for the \emph{blockwise} hypercube.}  Fix a block decomposition of the hypercube \(\{0,1\}^N=\{0,1\}^{n_1}\times\cdots\times\{0,1\}^{n_k}\).  We will work with the polynomial ring \(\mb{R}[\mb{X}]\), where \(\mb{X}=(\mb{X}_1,\ldots,\mb{X}_k)\), and \(\mb{X}_j\) is the set of indeterminates for the \(j\)-th block.

We are interested in covering subsets of the hypercube \(\{0,1\}^N\) by polynomials and families of hyperplanes.  In this context, our proof techniques work under some nondegeneracy conditions.  Let \(S\subsetneq\{0,1\}^N\), and consider \emph{multiplicity parameters} \(t\ge1,\,\ell\in[0,t-1]\).  We define
\begin{itemize}
	\item  a nonzero polynomial \(P(\mb{X})\in\mb{R}[\mb{X}]\) to be a \tsf{\((t,\ell)\)-block exact polynomial cover} for \(S\) if
	\begin{enumerate}[(a)]
		\item  for every point \(a\in S\), we have \(\mult(P(\mb{X}),a)\ge t\).
		\item  for each \(j\in[k]\), and every point \((a,\wt{a})\in\{0,1\}^N\setminus S\) with \(a\in\{0,1\}^{n_1}\times\cdots\times\{0,1\}^{n_{j-1}}\times\{0,1\}^{n_{j+1}}\times\cdots\times\{0,1\}^{n_k},\,\wt{a}\in\{0,1\}^{n_j}\), we have \(\mult(P(a,\mb{X}_j),\wt{a})=\ell\).
	\end{enumerate}

	\item  a finite multiset of hyperplanes \(\mc{H}(\mb{X})\) in \(\mb{R}^N\) to be a \tsf{\((t,\ell)\)-block exact hyperplane cover} for \(S\) if
	\begin{enumerate}[(a)]
		\item  for every \(a\in S\), we have \(|\{H(\mb{X})\in\mc{H}(\mb{X}):H(a)=0\}|\ge t\).
		\item  for every \(b\in\{0,1\}^N\setminus S\), we have \(|\{H(\mb{X})\in\mc{H}(\mb{X}):H(b)=0\}|=\ell\).
		\item  for each \(j\in[k]\), and every \(a\in\{0,1\}^{n_1}\times\cdots\times\{0,1\}^{n_{j-1}}\times\{0,1\}^{n_{j+1}}\times\cdots\times\{0,1\}^{n_k}\), we have \(|\mc{H}(a,\mb{X}_j)|=|\mc{H}(\mb{X})|\).
		
		(In other words, no two hyperplanes in the family \emph{collapse} into one, upon restriction to any single block.)
	\end{enumerate}
\end{itemize}
Let \(\bEHC_{(n_1,\ldots,n_k)}^{(t,\ell)}(S)\) denote the minimum size of a \((t,\ell)\)-block exact hyperplane cover for \(S\), and let \(\bEPC_{(n_1,\ldots,n_k)}^{(t,\ell)}(S)\) denote the minimum degree of a \((t,\ell)\)-block exact polynomial cover for \(S\).  The definitions immediately imply that every \((t,\ell)\)-block exact hyperplane cover is also a \((t,\ell)\)-block exact polynomial cover, and so \(\bEHC_{(n_1,\ldots,n_k)}^{(t,\ell)}(S)\ge\bEPC_{(n_1,\ldots,n_k)}^{(t,\ell)}(S)\).  For completeness, we give a quick proof in Appendix~\ref{app:bEHC-bEPC}.  Further, it is trivial that \(\bEHC_{(n_1,\ldots,n_k)}^{(t,\ell)}(S)\ge\EHC_N^{(t,\ell)}(S)\) and \(\bEPC_{(n_1,\ldots,n_k)}^{(t,\ell)}(S)\ge\EPC_N^{(t,\ell)}(S)\).

\subsection{Peripheral intervals, and inner and outer intervals of symmetric sets}\label{subsec:prelims-inner-outer-peripheral}

For any \(a\in[-1,n-1],\,b\in[1,n+1],\,a<b\), denote the set of weights \(I_{n,a,b}=[0,a]\cup[b,n]\), and we say a \tsf{peripheral interval} is the symmetric set \(J_{n,a,b}\subseteq\{0,1\}^n\) defined by \(W_n(J_{n,a,b})=I_{n,a,b}\).  Here, we have the convention \([0,-1]=[n+1,n]=\emptyset\).  In other words, a peripheral interval \(J_{n,a,b}\) could be either (i)\n\emph{one-sided}, that is, one or both of the weight intervals \([0,a],\,[b,n]\) could be empty (\(a=-1\) or \(b=n+1\) or both), or (ii)\n\emph{two-sided}, that is, both the weight intervals \([0,a],\,[b,n]\) are nonempty (\(a\ge1\) and \(b\le n\)).

Now let \(S\subseteq\{0,1\}^n\) be a symmetric set.
\begin{itemize}
	\item  If \(S\subsetneq\{0,1\}^n\), then the \tsf{inner interval} of \(S\), denoted by \(\innint(S)\), is defined to be the peripheral interval \(J_{n,a,b}\subseteq\{0,1\}^n\) of maximum size such that \(J_{n,a,b}\subseteq S\).  It is easy to check that \(\innint(S)\) is unique.  Further, we define \(\innint(\{0,1\}^n)=J_{n,\lfloor n/2\rfloor,\lfloor n/2\rfloor+1}\).
	
	\item  Let \(\mc{O}(S)\) be the collection of all peripheral intervals \(J_{n,a,b}\) such that \(S\subseteq J_{n,a,b}\) and \(I_{n,a,b}=W_n(J_{n,a,b})\) has minimum size.  	It is easy to see that \(\mc{O}(S)\) can contain several peripheral intervals; the following is an example.
	\begin{example}\label{ex:OS-nonunique}
		Let \(n\) be even, and choose \(W_n(S)=\{w\in[0,n]:w\tx{ is even}\}\).  Then for any even \(w\in[0,n]\), we have \(I_{n,w,w+2}=[0,w]\cup[w+2,n],\,|I_{n,w,w+2}|=n\) and \(S\subseteq J_{n,w,w+2}\).  Further, for any peripheral interval \(J_{n,a,b}\supseteq S\), it is immediate that \(|b-a|\le2\), and so \(|I_{n,a,b}|\ge n\).  Thus, \(\mc{O}(S)=\{J_{n,w,w+2}:w\in[0,n]\tx{ is even}\}\).
	\end{example}
	
	Moving on, consider the function \(\lambda_S:\mc{O}(S)\to\mb{N}\) defined by
	\[
	\lambda_S(J_{n,a,b})=|a+b-n|,\quad\tx{for all }J_{n,a,b}\in\mc{O}(S).
	\]
	It is easy to check that the minimizer of \(\lambda_S\) is either a unique peripheral interval \(J_{n,a,b}\), or exactly a pair of peripheral intervals \(\{J_{n,a,b},J_{n,n-b,n-a}\}\).  The \tsf{outer interval} of \(S\), denoted by \(\outint(S)\), is defined by
	\[
	\outint(S)=\begin{cases}
		J_{n,a,b}&\tx{if }J_{n,a,b}\tx{ is the unique minimizer of }\lambda_S,\\
		J_{n,a,b}&\tx{if }\{J_{n,a,b},J_{n,n-b,n-a}\}\tx{ are minimizers of }\lambda_S\tx{ and }a>n-b.
	\end{cases}
	\]
	Therefore, \(\outint(S)\) is unique.
\end{itemize}
We will discuss more on uniqueness of inner and outer intervals, and look at some illustrations, in Appendix~\ref{app:inner-outer}.  Now define
\begin{align*}
	\inn_n(S)&=(\min\{a,n-b\}+1)+|W_n(S)\setminus W_{n,\min\{a,n-b\}+1}|&&\tx{where }J_{n,a,b}=\innint(S),\\
	\tx{and }\out_n(S)&=a+n-b+1=|I_{n,a,b}|-1&&\tx{where }J_{n,a,b}=\outint(S).
\end{align*}
\begin{remark}\label{rem:inner-outer-peripheral}
	Let \(J_{n,a,b}\subseteq\{0,1\}^n\) be a peripheral interval.  It is trivially true that \(\innint(J_{n,a,b})=\outint(J_{n,a,b})=J_{n,a,b}\).  Further, it is easy to check that
	\begin{align*}
		\innint(\{0,1\}^n\setminus J_{n,a,b})&=\begin{cases}
			\emptyset&\tx{if }a\ge1,\,b\le n,\\
			J_{n,a+1,n}&\tx{if }b=n+1,\\
			J_{n,0,b-1}&\tx{if }a=-1,
		\end{cases}\\
		\tx{and}\quad\outint(\{0,1\}^n\setminus J_{n,a,b})&=\begin{cases}
			J_{n,-1,a+1}&\tx{if }a\ge n-b,\\
			J_{n,b-1,n+1}&\tx{if }a<n-b.
		\end{cases}
	\end{align*}
	Therefore,
	\begin{align*}
		\inn_n(J_{n,a,b})&=\max\{a,n-b\}+1,&\inn_n(\{0,1\}^n\setminus J_{n,a,b})&=b-a-1,\\
		\out_n(J_{n,a,b})&=a+n-b+1,&\out_n(\{0,1\}^n\setminus J_{n,a,b})&=\min\{n-a,b\}-1.
	\end{align*}
\end{remark}
\noindent The following interesting and important observations are immediate from Remark~\ref{rem:inner-outer-peripheral}, and the definitions.
\begin{observation}\label{obs:peripheral-tight}
	\begin{enumerate}[(a)]
		\item  For any peripheral interval \(J_{n,a,b}\subseteq\{0,1\}^n\), we have
		\[
		\inn_n(J_{n,a,b})+\out_n(\{0,1\}^n\setminus J_{n,a,b})=\inn_n(\{0,1\}^n\setminus J_{n,a,b})+\out_n(J_{n,a,b})=n.
		\]
		\item  For any symmetric set \(S\subseteq\{0,1\}^n\), we have \(S=\innint(S)=\outint(S)\) if and only if either \(S\) or \(\{0,1\}^n\setminus S\) is a peripheral interval.
	\end{enumerate}
	
\end{observation}

%\begin{itemize}
%	\item For any $n\in\NN$, $[n]$ denotes the set $\{1,2,\dots,n\}.$
%	
%	\item For any $p,n\in\NN$ with $p\leq n$, $\binom{[n]}{p}$ denotes the collection of all subsets of $[n]$ with size $p$.
%	
%	\item For any $x=(x_1,\dots,x_n)\in\RR^n$ and $I=\{i_1,\dots, i_p\}\subset [n]$ with $i_1<i_2<\dots<i_p$, $x\big|_I$ denotes the point $(x_{i_1},x_{i_2},\dots, x_{i_p})\in\RR^p$.
%	
%	
%	\item For any $u=(u_1,\dots,u_n)\in\cQ^n$, $\mathrm{wt}(u):=\sum_{i=1}^n u_i$.
%	
%	\item For any symmetric subset $S$ of $\cQ^{n}$, $W_n(S) := \left\{\mathrm{wt}(u)\;|\;u\in S\right\}$.
%	
%	\item For any $i \in \{0,\dots, n\}$, $W_{n,i}$ denotes the set
%	$
%	W_{n,i} := \left\{ 0, \dots, i-1\right\} \cup
%	\left\{ n-i+1, \dots, n\right\}.
%	$
%	Observe that $W_{n,0}=\emptyset$.
%	
%	\item   
%	For any symmetric set $S\subset\cQ^n,\;\Lambda_{n}(S)=|W_n(S)|-\max\big\{i\in\{0,1,\dots,n\}\;\big|\;W_{n,i}\subseteq W_n(S)\big\}$ 
%\end{itemize}

%\input{parsym.tex}

\section{Index complexity of symmetric and PDC \(k\)-wise symmetric sets}\label{sec:index-complexity-PDC}

\subsection{Inner and outer intervals of symmetric sets}\label{subsec:inner-outer-symmetric}

Let us first prove Proposition~\ref{pro:inner-outer}, which relates the inner and outer intervals of symmetric sets.  We will give two proofs, one combinatorial and another via the polynomial method.  We mention the statement again, for convenience.  
\begin{repproposition}[\ref{pro:inner-outer}]
	For any nonempty symmetric set \(S\subseteq\{0,1\}^n\), we have
	\[
	\inn_n(\{0,1\}^n\setminus S)+\out_n(S)\ge n.
	\]
	Further, equality holds if and only if either \(S\) or \(\{0,1\}^n\setminus S\) is a peripheral interval.
\end{repproposition}
\begin{proof}[First proof]
	We note that the assertion is immediately true, by Observation~\ref{obs:peripheral-tight}(a), if either \(S\) or \(\{0,1\}^n\setminus S\) is a peripheral interval.
	
	Now suppose \(S\subseteq\{0,1\}^n\) is some nonempty symmetric set.  Let \(J_{n,a,b}=\outint(S)\).  So by definition, we get \(\out_n(S)=\out_n(J_{n,a,b})\).  It is, therefore, enough to prove \(\inn_n(\{0,1\}^n\setminus S)\ge\inn_n(\{0,1\}^n\setminus J_{n,a,b})\).
	
	Let \(J_{n,a',b'}=\innint(\{0,1\}^n\setminus S)\), and
	\[
	M_n(S)\coloneqq\{w\in W_n(\{0,1\}^n\setminus S):\min\{a',n-b'\}+1\le w\le\max\{n-a',b'\}-1\}.
	\]
	So, by definition, \(\inn_n(\{0,1\}^n\setminus S)=\min\{a',n-b'\}+1+|M_n(S)|\).  Also, since \(J_{n,a,b}=\outint(S)\), we get
	\[
	[0,n]\setminus I_{n,a,b}\subseteq I_{n,\min\{a',n-b'\},\max\{n-a',b'\}}\sqcup M_n(S).
	\]
	This immediately gives
	\begin{align}
	\inn_n(\{0,1\}^n\setminus J_{n,a,b})&\le|I_{n,\min\{a',n-b'\},\max\{n-a',b'\}}|+|M_n(S)|\notag\\
	&=\min\{a',n-b'\}+1+|M_n(S)|\notag\\
	&=\inn_n(\{0,1\}^n\setminus S).\label{eq:innint}
	\end{align}
	
	It is clear that equality is attained exactly when \(\inn_n(\{0,1\}^n\setminus S)=\inn_n(\{0,1\}^n\setminus J_{n,a,b})\).  By~(\ref{eq:innint}), this means equality is attained exactly when \(\inn_n(\{0,1\}^n\setminus J_{n,a,b})=\min\{a',n-b'\}+1+|M_n(S)|\).  This happens exactly when \(S=J_{n,a,b}=J_{n,a',b'}\), that is, \(S=\innint(S)=\outint(S)\).  By Observation~\ref{obs:peripheral-tight}(b), this is equivalent to either \(S\) or \(\{0,1\}^n\setminus S\) being a peripheral interval.
\end{proof}

\begin{proof}[Second proof]
	Let \(J_{n,a,b}=\outint(S)\).  By the minimality of size of \(I_{n,a,b}\), we have \(\{a,b\}\subseteq W_n(S)\).  Without loss of generality, assume \(a\ge n-b\).  Define \(P(\mb{X})=X_1\cdots X_a(X_{a+1}-1)\cdots(X_{a+n-b+1}-1)\).  We clearly have \(1^a0^{n-a}\in S\), and \(P(1^a0^{n-a})\ne0\).  Now consider any \(x\in S,\,x\ne1^a0^{n-a}\).  We have three cases.
	\begin{enumerate}[(C1)]
		\item  \(|x|=a,\,x\ne1^a0^{n-a}\).  Then there exists \(i\in[1,a]\) such that \(x_i=0\), so and \(P(x)=0\).
		\item  \(|x|<a\).  Then there exists \(i\in[1,a]\) such that \(x_i=0\), and s \(P(x)=0\).
		\item  \(|x|>a\), which means \(|x|\ge b\), since \(S\subseteq J_{n,a,b}\).  So \(|\{i\in[n]:x_i=0\}|<n-b+1\).  This implies that there exists \(i\in[a+1,a+n-b+1]\) such that \(x_i=1\), and so \(P(x)=0\).
	\end{enumerate}
	Consider the family of hyperplanes
	\[
	\msf{h}(\mb{X})\coloneqq\mc{H}^*_{\ol{\mu}_n(S)}(\mb{X})\sqcup\mc{H}'_{W_n(\{0,1\}^n\setminus S)\setminus W_{n,\ol{\mu}_n(S)}}(\mb{X}).
	\]
	By Lemma~\ref{lem:T}, we have \(\mc{H}^*_{\ol{\mu}_n(S)}(x)=0\) if and only if \(|x|\in W_{n,\ol{\mu}_n(S)}\).  Further, by definition, we have \(\mc{H}'_{W_n(\{0,1\}^n\setminus S)\setminus W_{n,\ol{\mu}_n(S)}}(x)=0\) if and only if \(|x|\in W_n(\{0,1\}^n\setminus S)\setminus W_{n,\ol{\mu}_n(S)}\).  Thus, we have \(\msf{h}(x)=0\) if and only if \(x\in\{0,1\}^n\setminus S\).
	
	So we conclude that the polynomial \(P\msf{h}(\mb{X})\) satisfies \(P\msf{h}(1^a0^{n-a})\ne0\), and \(P\msf{h}(x)=0\) for all \(x\in\{0,1\}^n,\,x\ne1^a0^{n-a}\).  Therefore, by Theorem~\ref{thm:alon-furedi-again}, we get \(\deg(P\msf{h})\ge n\).  Now by the definitions, we also have
	\begin{align*}
		\deg(P)&=a+n-b+1=\out_n(S),\\
		\tx{and}\quad\deg(\msf{h})&=\ol{\mu}_n(S)+|W_n(\{0,1\}^n\setminus S)\setminus W_{n,\ol{\mu}_n(S)}|=\inn_n(\{0,1\}^n\setminus S).
	\end{align*}
	Hence,
	\[
	\inn_n(\{0,1\}^n\setminus S)+\out_n(S)=\deg(\msf{h})+\deg(P)=\deg(P\msf{h})\ge n.
	\]
	
	By the definition above, \(\deg(P)=a+n-b+1=\out_n(S)\), and therefore we have shown that \(\deg(\msf{h})\ge b-a-1\).  But, again by the definition above, we have \(\deg(\msf{h})=\ol{\mu}_n(S)+|W_n(\{0,1\}^n\setminus S)\setminus W_{n,\ol{\mu}_n(S)}|\).  Thus, we have equality exactly when \(\ol{\mu}_n(S)+|W_n(\{0,1\}^n\setminus S)\setminus W_{n,\ol{\mu}_n(S)}|=b-a-1\), where \(J_{n,a,b}=\outint(S)\).  This is true if and only if \(S=J_{n,a,b}\).  By Observation~\ref{obs:peripheral-tight}(b), this is equivalent to either \(S\) or \(\{0,1\}^n\setminus S\) being a peripheral interval.
\end{proof}

\subsection{Index complexity of symmetric sets}\label{subsec:main-index-complexity-symmetric}

We will now proceed to prove Proposition~\ref{pro:index-complexity-symmetric} which characterizes the index complexity of symmetric sets.  We will need a definition and a technical lemma.  Let \(p\in\{0,1\}^n\) and denote \(I_0(p)\coloneqq\{i\in[n]:p_i=0\},\,I_1(p)\coloneqq\{i\in[n]:p_i=1\}\).  So \(|I_0(p)|=n-|p|\) and \(|I_1(p)|=|p|\).  For any \(I_0\subseteq I_0(p),\,I_1\subseteq I_1(p)\), we define the \tsf{separation of \(p\) with respect to \((I_0,I_1)\)}, denoted by \(\sep(p,I_0,I_1)\subseteq\{0,1\}^n\), to be the maximal symmetric set such that for every \(x\in\sep(p,I_0,I_1)\), we have \(x_{I_0\sqcup I_1}\ne p_{I_0\sqcup I_1}\).  We will refer to the special case \(\sep(p)\coloneqq\sep(p,I_0(p),I_1(p))\) as simply the \tsf{separation of \(p\)}.
\begin{remark}\label{rem:sep}
	It follows by definition that \(|p|\not\in W_n(\sep(p,I_0,I_1))\), for any \(I_0\subseteq I_0(p),\,I_1\subseteq I_1(p)\).
\end{remark}
\begin{lemma}\label{lem:sym-index}
	For any \(p\in\{0,1\}^n\), and \(I_0\subseteq I_0(p),\,I_1\subseteq I_1(p)\), we have
	\[
	\sep(p,I_0,I_1)=J_{n,|I_1|-1,n-|I_0|+1}.
	\]
	In particular, we have \(\sep(p)=J_{n,|p|-1,|p|+1}\).
\end{lemma}
\begin{proof}
	Without loss of generality, assume \(p=1^a0^{n-a}\), and \(I_1=[1,u]\subseteq[1,a],\,I_0=[n-v+1,n]\), for some \(u\in[0,a],\,v\in[0,n-a]\).  So \(|I_1|=u,\,|I_0|=v\).  We observe the following.
	\begin{enumerate}[(P1)]
		\item  For any \(x=1^{a'}y\) with \(a'\ge u\) and \(y\in\{0,1\}^{n-a'}\), we have \(x_{I_1}=p_{I_1}=1^u\).
		\item  For any \(x=y0^{b'}\) with \(b'\ge v\) and \(y\in\{0,1\}^{n-b'}\), we have \(x_{I_0}=p_{I_0}=0^v\).
	\end{enumerate}
	Combining the above two observations, we get that for any \(x=1^{a'}y0^{b'}\) with \(a'\ge u,\,b'\ge v\) and \(y\in\{0,1\}^{n-a'-b'}\), we have \(x_{I_0\sqcup I_1}=p_{I_0\sqcup I_1}\).  Since \(\sep(p,I_0,I_1)\) is a symmetric set, this implies that
	\[
	[u,n-v]\cap W_n(\sep(p,I_0,I_1))=\emptyset,\quad\tx{that is},\quad \sep(p,I_0,I_1)\subseteq J_{n,u-1,n-v+1}.
	\]
	Now consider any \(x\in J_{n,u-1,n-v+1}\).  We have two cases.
	\begin{enumerate}[(C1)]
		\item  \(|x|\le u-1\).  Since \(|I_1|=u\), there exists \(i\in I_1\) such that \(x_i=0\), but \(p_i=1\).
		\item  \(|x|\ge n-v+1\).  Since \(|I_0|=v\), there exists \(i\in I_0\) such that \(x_i=1\), but \(p_i=0\).
	\end{enumerate}
	Hence, we conclude that \(J_{n,u-1,n-v+1}\subseteq\sep(p,I_0,I_1)\).
\end{proof}

\noindent We are now ready to prove Proposition~\ref{pro:index-complexity-symmetric}.  We mention the statement again, for convenience.
\begin{repproposition}[\ref{pro:index-complexity-symmetric}]
	For any nonempty symmetric set \(S\subseteq\{0,1\}^n\), we have \(r_n(S)=\out_n(S)\).
\end{repproposition}
\begin{proof}
	Let \(J_{n,a,b}\) be the outer interval of \(S\).  So \(\out_n(S)=a+n-b+1\).  By the minimality of size of \(I_{n,a,b}\), we have \(\{a,b\}\subseteq W_n(S)\).  So, in particular, \(p\coloneqq 1^a0^{n-a}\in S\).  Without loss of generality, assume \(a\ge n-b\).  Now consider any \(x\in S,\,x\ne1^a0^{n-a}\).  We have three cases.
	\begin{enumerate}[(C1)]
		\item  \(|x|=a,\,x\ne1^a0^{n-a}\).  Then there exists \(i\in[1,a]\) such that \(x_i=0\), but \(p_i=1\).
		\item  \(|x|<a\).  Then there exists \(i\in[1,a]\) such that \(x_i=0\), but \(p_i=1\).
		\item  \(|x|>a\), which means \(|x|\ge b\), since \(S\subseteq J_{n,a,b}\).  So \(|\{i\in[n]:x_i=0\}|<n-b+1\).  This implies there exists \(i\in[a+1,a+n-b+1]\) such that \(x_i=1\), but \(p_i=0\).
	\end{enumerate}
	Thus, in all three cases, there exists \(i\in[1,a+n-b+1]\) such that \(x_i\ne p_i\).  Hence, we conclude that \(r_n(S)\le a+n-b+1=\out_n(S)\).
	
	Now, in order to prove the reverse inequality, let \(p\in S\) and \(I\subseteq[n],\,r_n(S)=|I|\) such that for every \(x\in S,\,x\ne p\), we have \(x_I\ne p_I\).  Further, let \(I_0=I\cap I_0(p)\) and \(I_1=I\cap I_1(p)\).  By definition of index complexity, Lemma~\ref{lem:sym-index}, and Remark~\ref{rem:sep}, we get
	\begin{align}
	W_n(S)\setminus\{|p|\}\subseteq W_n(\sep(p,I_0,I_1))=I_{n,|I_1|-1,n-|I_0|+1}.\label{sep-containment}
	\end{align}
	Also, trivially, we have \(|I_1|\le|p|\le n -|I_0|\).  Since \(J_{n,a,b}\) is the outer interval of \(S\), we have exactly one of the two cases, by~(\ref{sep-containment}) and the minimality of size of \(I_{n,a,b}\).
	\begin{enumerate}[(C1')]
		\item  \(|p|=a\) and \(W_n(S)\setminus\{p\}\subseteq I_{n,a-1,b}\subseteq I_{n,|I_1|-1,n-|I_0|+1}\).  So \(|I_1|\ge a,\,|I_0|\ge n-b+1\).
		\item  \(|p|=b\) and \(W_n(S)\setminus\{p\}\subseteq I_{n,a,b+1}\subseteq I_{n,|I_1|-1,n-|I_0|+1}\).  So \(|I_1|\ge a+1,\,|I_0|\ge n-b\).
	\end{enumerate}
	In either of the two cases, we finally get
	\[
	r_n(S)=|I|=|I_0|+|I_1|\ge a+n-b+1=\out_n(S).\qedhere
	\]
\end{proof}
\begin{remark}\label{rem:index-complexity-symmetric}
	We also note from the proof of Proposition~\ref{pro:index-complexity-symmetric}, that if \(J_{n,a,b}\) is the outer interval of the symmetric set \(S\), with \(a\ge n-b\), then the set \(I=[1,a+n-b+1]\) satisfies \(|I|=r_n(S)\), and the point \(p=1^a0^{n-a}\) is such that for every \(x\in S,\,x\ne1^a0^{n-a}\), we have \(x_I\ne(1^a0^{n-a})_I\).  On the other hand, if \(a<n-b\), then these choices change to \(I=[b-a,n]\) and \(p=1^b0^{n-b}\).
\end{remark}

\subsection{Index complexity of PDC \(k\)-wise symmetric sets}\label{subsec:main-index-complexity-PDC}

Let us now proceed to prove Proposition~\ref{pro:PDC-index-complexity}, which characterizes the index complexity of PDC \(k\)-wise symmetric sets.  We mention the statement again, for convenience.
\begin{repproposition}[\ref{pro:PDC-index-complexity}]
	For any nonempty outer intact PDC \(k\)-wise symmetric set \(S\subseteq\{0,1\}^N\), we have
	\[
	r_N(S)=\sum_{j=1}^kr_{n_j}(S_j)=\sum_{j=1}^k\out_{n_j}(S_j).
	\]
\end{repproposition}
\begin{proof}
	For each \(j\in[k]\), let \(J_{n_j,a_j,b_j}\) be the outer interval of \(S_j\), and further, as indicated in Remark~\ref{rem:index-complexity-symmetric}, let
	\[
	\big(p^{(j)},I^{(j)}\big)=\begin{cases}
		\big(1^{a_j}0^{n_j-a_j},[1,a_j+n_j-b_j+1]\big)&\tx{if }a_j\ge n_j-b_j,\\
		\big(1^{b_j}0^{n_j-b_j},[b_j-a_j,n_j]\big)&\tx{if }a_j<n_j-b_j,
	\end{cases}
	\]
	satisfy the definition of index complexity \(r_{n_j}(S_j)\).  Now consider any \((z_1,\ldots,z_k)\in\innext(S)\).  Since \(S\) is outer intact, for each \(j\in[k]\), we have the following.
	\begin{itemize}
		\item  \(J_{n_j,a_j,b_j}\) is the outer interval of \([S]_{j,z_j}\).
		\item  \(p^{(j)}\in[S]_{j,z_j}\).
		\item  \(p^{(j)}\) and \(I^{(j)}\) satisfy the definition of index complexity \(r_{n_j}([S]_{j,z_j})\), as indicated in Remark~\ref{rem:index-complexity-symmetric}.
	\end{itemize}
	Define \(p=(p^{(1)},\ldots,p^{(k)})\in S\) and \(I=I^{(1)}\sqcup\cdots\sqcup I^{(k)}\).  Now consider any \(x=(x^{(1)},\ldots,x^{(k)})\in S,\,x\ne p\).  So there exists \(j\in[k]\) such that \(x^{(j)}\ne p^{(j)}\).  Since \(x\in S\), there exists \((z_1,\ldots,z_k)\in\innext(S)\) such that \(x\in[S]_{1,z_1}\times\cdots\times[S]_{k,z_k}\), and so \(x^{(j)}\in[S]_{j,z_j}\).  Then by the choice of \(I^{(j)}\), we get \(x^{(j)}_{I^{(j)}}\ne p^{(j)}_{I^{(j)}}\).  Thus, we have \(r_N(S)\le|I|=\sum_{j=1}^k{I^{(j)}}|=\sum_{j=1}^kr_{n_j}(S_j)\).
	
	To prove the reverse inequality, now suppose \(p=(p^{(1)},\ldots,p^{(k)})\in S\) and \(I=I^{(1)}\sqcup\cdots\sqcup I^{(k)}\subseteq[N]\) satisfy the definition of index complexity \(r_N(S)\).  Let \((z_1,\ldots,z_k)\in\innext(S)\) such that \(p\in[S]_{1,z_1}\times\cdots\times[S]_{k,z_k}\).  Fix any \(j\in[k]\).  Consider any \(y\in[S]_{j,z_j},\,y\ne p^{(j)}\), and let \(x=(p^{(1)},\ldots,p^{(j-1)},y,p^{(j+1)},\ldots,p^{(k)})\).  Then \(x\in[S]_{1,z_1}\times\cdots\times[S]_{k,z_k}\subseteq S\) and \(x\ne p\).  This implies \(x_I\ne p_I\), which means \(y_{I^{(j)}}\ne p^{(j)}_{I^{(j)}}\).  Thus, we get \(|I^{(j)}|\ge r_{n_j}([S]_{j,z_j})=r_{n_j}(S_j)\).  Hence, \(r_N(S)=|I|=\sum_{j=1}^k|I^{(j)}|\ge\sum_{j=1}^kr_{n_j}(S_j)\).
	
	The final equality in the statement is then immediate from Proposition~\ref{pro:index-complexity-symmetric}.
\end{proof}

\section{Covering PDC \(k\)-wise symmetric sets}\label{sec:covering-PDC}

Let us now prove our third main result (Theorem~\ref{thm:EPC-PDC}).  We mention the statement again, for convenience.  Recall that we work with the indeterminates \(\mb{X}=(\mb{X}_1,\ldots,\mb{X}_k)\), where \(\mb{X}_j=(X_{j,1},\ldots,X_{j,n_j})\) are the indeterminates for the \(j\)-th block.
\begin{reptheorem}[\ref{thm:EPC-PDC}]
	For any nonempty PDC \(k\)-wise symmetric set \(S\subseteq\{0,1\}^N\) and \(t\ge1\), we have
	\[
	\bEPC_{(n_1,\ldots,n_k)}^{(t,t-1)}(\{0,1\}^N\setminus S)=\max_{(z_1,\ldots,z_k)\in\outext(S)}\bigg\{\sum_{j\in[k]:z_j\ge1} \ol{\Lambda}_{n_j}([S]_{j,z_j-1})\bigg\}+2t-2.
	\]
\end{reptheorem}
\begin{proof}
	Let us first prove the lower bound.  Let \(P(\mb{X})\in\mb{R}[\mb{X}]\) be a \((t,t-1)\)-block exact polynomial cover for \(\{0,1\}^N\setminus S\).  Fix any \((z_1,\ldots,z_k)\in\outext(S)\).  Note that for any \(j\in[k]\), we have \(z_j\ge1\) if and only if \([S]_{j,z_j-1}\ne\emptyset\).  So, without loss of generality, we assume \(z_j\ge1\) for all \(j\in[k]\).  It is now enough to show that
	\[
	\deg(P)\ge\sum_{j=1}^k\ol{\Lambda}_{n_j}([S]_{j,z_j-1})+2t-2.
	\]
	Consider any \(j\in[k]\), and let
	\[
	\ol{\mu}_j\coloneqq\ol{\mu}_n([S]_{j,z_j-1})=\max\{i\in[0,\lceil n/2\rceil]:W_{n_j,i}\subseteq W_n(\{0,1\}^{n_j}\setminus[S]_{j,z_j-1})\}.
	\]
	So either \(\ol{\mu}_j\in W_n([S]_{j,z_j-1})\) or \(n_j-\ol{\mu}_j\in W_n([S]_{j,z_j-1})\).  Without loss of generality, suppose \(\ol{\mu}_j\in W_n([S]_{j,z_j-1})\).  Then clearly, \(|W_{n_j}([S]_{j,z_j-1})\setminus\{\ol{\mu}_j\}|=n_j-|W_n(\{0,1\}^{n_j}\setminus[S]_{j,z_j-1})|\).  Also clearly, \(1^{\ol{\mu}_j}0^{n_j-\ol{\mu}_j}\in[S]_{j,z_j-1}\).  Define
	\[
	Q(\mb{X})=P(\mb{X})\cdot\prod_{j=1}^k\mc{H}'_{W_{n_j}([S]_{j,z_j-1})\setminus\{\ol{\mu}_j\}}(\mb{X}_j).
	\]
	
	Recall that we have \(W_{(n_1,\ldots,n_k)}(S)=\{(|x^{(1)}|,\ldots,|x^{(k)}|):x\in S\}\).  Further, recall that we have \(W_{(n_1,\ldots,n_k)}([S]_{1,z_1-1}\times\cdots\times[S]_{k,z_k-1})\coloneqq W_{n_1}([S]_{1,z_1-1})\times\cdots\times W_{n_k}([S]_{k,z_k-1})\).  Consider any \(x=(x^{(1)},\ldots,x^{(k)})\in\{0,1\}^N\).
	We have the following cases.
	\begin{enumerate}[(C1)]
		\item  \((|x^{(1)}|,\ldots,|x^{(k)}|)=(\ol{\mu}_1,\ldots,\ol{\mu}_k)\).  So we have
		\begin{itemize}
			\item  \(\mult(P(x'_{(j)},\mb{X}_j),x^{(j)})=t-1\), where \(x=(x'_{(j)},x^{(j)})\), for every \(j\in[k]\).
			\item  \(\mc{H}'_{W_{n_j}([S]_{j,z_j-1})\setminus\{\ol{\mu}_j\}}(x^{(j)})\ne0\), for every \(j\in[k]\).
		\end{itemize}
		This implies \(\mult(Q(\mb{X}),x)=t-1\).  Note that this is where we need \(P(\mb{X})\) to be a \((t,t-1)\)-block exact polynomial cover and not just a \((t,t-1)\)-exact polynomial cover for \(\{0,1\}^N\setminus S\).
		
%		Further, we have \(\big(X_1\cdots X_{\ol{\mu}_j}\big)(x^{(j)})=0\) if and only if \(x^{(j)}\ne1^{\ol{\mu}_j}0^{n_j-\ol{\mu}_j}\).  This implies
%		\[
%		\mult(Q(\mb{X}),x)\begin{cases}
%			=t-1&\tx{if }x\ne(1^{\ol{\mu}_1}0^{n_1-\ol{\mu}_1},\ldots,1^{\ol{\mu}_k}0^{n_k-\ol{\mu}_k}),\\
%			\ge t&\tx{if }x=(1^{\ol{\mu}_1}0^{n_1-\ol{\mu}_1},\ldots,1^{\ol{\mu}_k}0^{n_k-\ol{\mu}_k}).
%		\end{cases}
%		\]
		
		\item  \((|x^{(1)}|,\ldots,|x^{(k)}|)\in W_{(n_1,\ldots,n_k)}([S]_{1,z_1-1}\times\cdots\times[S]_{k,z_k-1})\setminus\{(\ol{\mu}_1,\ldots,\ol{\mu}_k)\}\).  So we have
		\begin{itemize}
			\item  \(\mult(P(x'_{(j)},\mb{X}_j),x^{(j)})=t-1\), where \(x=(x'_{(j)},x^{(j)})\), for every \(j\in[k]\).
			\item  There exists \(j\in[k]\) such that \(|x^{(j)}|\ne\ol{\mu}_j\), and so \(\mc{H}'_{W_{n_j}([S]_{j,z_j-1})\setminus\{\ol{\mu}_j\}}(x^{(j)})=0\)
		\end{itemize}
		This implies \(\mult(Q(\mb{X}),x)\ge t\).
		
		\item  \((|x^{(1)}|,\ldots,|x^{(k)}|)\not\in W_{(n_1,\ldots,n_k)}([S]_{1,z_1-1}\times\cdots\times[S]_{k,z_k-1})\).  So \(\mult(P(\mb{X}),x)\ge t\), and this implies \(\mult(Q(\mb{X}),x)\ge t\).
	\end{enumerate}
	Thus, \(Q(\mb{X})\) is a \((t,t-1)\)-exact polynomial cover for \(\{0,1\}^N\setminus L\), where \(L=L_1\times\cdots\times L_k\) is a \(k\)-wise layer given by \(W_{n_j}(L_j)=\{\ol{\mu}_j\},\,j\in[k]\).  So Theorem~\ref{thm:EPC-index} and Corollary~\ref{cor:k-layer-index-complexity} imply
	\begin{align}
	\deg(Q)\ge N-r_N(L)+2t-2=N-\sum_{j=1}^k\ol{\mu}_j+2t-2.\label{Q-lb}
	\end{align}
	Further, by construction, we have
	\begin{align}
	\deg(Q)&=\deg(P)+\sum_{j=1}^k\big(n_j-|W_n(\{0,1\}^{n_j}\setminus[S]_{j,z_j-1})|\big)\notag\\
	&=\deg(P)+N-\sum_{j=1}^k|W_{n_j}(\{0,1\}^{n_j}\setminus[S]_{j,z_j-1})|.\label{Q-ub}
	\end{align}
	From (\ref{Q-lb}) and (\ref{Q-ub}), we get
	\[
	\deg(P)\ge\sum_{j=1}^k\big(|W_{n_j}(\{0,1\}^{n_j}\setminus[S]_{j,z_j-1})|-\ol{\mu}_j\big)+2t-2=\sum_{j=1}^k\ol{\Lambda}_{n_j}([S]_{j,z_j-1})+2t-2.
	\]
	This completes the proof of the lower bound.
	
	Let us now show that the construction in Example~\ref{ex:EPC-PDC} attains the lower bound we just proved.  Recall that Example~\ref{ex:EPC-PDC} defines a polynomial
	\[
	\msf{h}_S(\mb{X})\coloneqq\bigg(\sum_{(z_1,\ldots,z_k)\in\outext(S)}\lambda_{S,(z_1,\ldots,z_k)}\mc{H}_{S,(z_1,\ldots,z_k)}(\mb{X})\bigg)\cdot\mc{H}^{\circ(t-1)}(\mb{X}_1),
	\]
	where, for each \((z_1,\ldots,z_k)\in\outext(S)\), we have
	\[
	\mc{H}_{S,(z_1,\ldots,z_k)}(\mb{X})=\prod_{j\in[k]:z_j\ge1}\Big(\mc{H}^*_{\ol{\mu}_n([S]_{j,z_j-1})}(\mb{X}_j)\cdot\mc{H}'_{W_{n_j}(\{0,1\}^{n_j}\setminus[S]_{j,z_j-1})\setminus W_{n_j,\ol{\mu}_n([S]_{j,z_j-1})}}(\mb{X}_j)\Big),
	\]
	and further, \(\{\lambda_{S,(z_1,\ldots,z_k)}:(z_1,\ldots,z_k)\in\outext(S)\}\subseteq\mb{R}\) is a \(\wh{\mb{Q}}\)-linearly independent subset of \(\mb{R}\), where the subfield \(\wh{\mb{Q}}\coloneqq\mb{Q}\big(\mc{H}_{S,(z_1,\ldots,z_k)}(b):b\in\{0,1\}^N,\,(z_1,\ldots,z_k)\in\outext(S)\big)\).  Firstly, note that since \(\mc{H}^{\circ(t-1)}(\mb{X}_1)=X_1^{t-1}(X_1-1)^{t-1}\), we clearly get \(\mult(\mc{H}^{\circ(t-1)}(\mb{X}_1),x)=t-1\) for all \(x\in\{0,1\}^N\).  Now fix any \((z_1,\ldots,z_k)\in\outext(S)\), and consider any \(x\in\{0,1\}^N\).  We observe
	\begin{align*}
		&\mc{H}_{S(z_1,\ldots,z_k)}(x)\ne0\\
		\iff\quad&\mc{H}^*_{\ol{\mu}_n([S]_{j,z_j-1})}(x^{(j)})\cdot\mc{H}'_{W_n(\{0,1\}^{n_j}\setminus[S]_{j,z_j-1})\setminus W_{n_j,\ol{\mu}_n([S]_{j,z_j-1})}}(x^{(j)})\ne0,&&\tx{for all }j\in[k]:z_j\ge1\\
		\iff\quad&x^{(j)}\not\in\{0,1\}^{n_j}\setminus[S]_{j,z_j-1},&&\tx{for all }j\in[k]:z_j\ge1\\
		\iff\quad&x\in\bigg(\prod_{j\in[k]:z_j\ge1}[S]_{j,z_j-1}\bigg)\times\bigg(\prod_{j\in[k]:z_j=0}\{0,1\}^{n_j}\bigg).
	\end{align*}
	Now it is easy to check that
	\[
	\{0,1\}^N\setminus S=\bigcap_{(z_1,\ldots,z_k)\in\outext(S)}\bigg(\{0,1\}^N\mathbin{\bigg\backslash}\bigg(\prod_{j\in[k]:z_j\ge1}[S]_{j,z_j-1}\bigg)\times\bigg(\prod_{j\in[k]:z_j=0}\{0,1\}^{n_j}\bigg)\bigg).
	\]
	So by the \(\wh{Q}\)-linear independence of \(\{\lambda_{S,(z_1,\ldots,z_k)}:(z_1,\ldots,z_k)\in\outext(S)\}\subseteq\mb{R}\), we get
	\begin{align*}
		&\sum_{(z_1,\ldots,z_k)\in\outext(S)}\lambda_{S,(z_1,\ldots,z_k)}\mc{H}_{S,(z_1,\ldots,z_k)}(x)=0\\
		\iff\quad&\mc{H}_{S,(z_1,\ldots,z_k)}(x)=0,\qquad\qquad\qquad\qquad\qquad\qquad\qquad\quad\,\tx{for all }(z_1,\ldots,z_k)\in\outext(S)\\
		\iff\quad&x\not\in\bigg(\prod_{j\in[k]:z_j\ge1}[S]_{j,z_j-1}\bigg)\times\bigg(\prod_{j\in[k]:z_j=0}\{0,1\}^{n_j}\bigg),\qquad\tx{for all }(z_1,\ldots,z_k)\in\outext(S)\\
		\iff\quad&x\in\bigcap_{(z_1,\ldots,z_k)\in\outext(S)}\bigg(\{0,1\}^N\mathbin{\bigg\backslash}\bigg(\prod_{j\in[k]:z_j\ge1}[S]_{j,z_j-1}\bigg)\times\bigg(\prod_{j\in[k]:z_j=0}\{0,1\}^{n_j}\bigg)\bigg)\\
		\iff\quad&x\in\{0,1\}^N\setminus S.
	\end{align*}
	Thus, we have
	\begin{itemize}
		\item  \(\mult(\msf{h}_S(\mb{X}),x)\ge t\) if \(x\in\{0,1\}^N\setminus S\).
		\item  \(\mult(\msf{h}_S(x_{(j)},\mb{X}_j),x^{(j)})=t-1\), where \(x=(x_{(j)},x^{(j)})\), for each \(j\in[k]\).
	\end{itemize}
	Hence, \(\msf{h}_S(\mb{X})\) is a \((t,t-1)\)-block exact polynomial cover for \(\{0,1\}^N\setminus S\).
	
	Now for any \((z_1,\ldots,z_k)\in\outext(S)\), we have
	\begin{align*}
	\deg(\mc{H}_{S,(z_1,\ldots,z_k)})&=\sum_{j\in[k]:z_j\ge1}\bigg(\ol{\mu}_n([S]_{j,z_j-1})+|W_{n_j}(\{0,1\}^{n_j}\setminus[S]_{j,z_j-1})\setminus W_{n_j,\ol{\mu}_n([S]_{j,z_j-1})}|\bigg)\\
	&=\sum_{j\in[k]:z_j\ge1}\bigg(W_{n_j}(\{0,1\}^{n_j}\setminus[S]_{j,z_j-1})|-\ol{\mu}_n([S]_{j,z_j-1})\bigg)\\
	&=\sum_{j\in[k]:z_j\ge1}\ol{\Lambda}_{n_j}([S]_{j,z_j-1}).
	\end{align*}
	Hence,
	\[
	\deg(\msf{h}_S)=\max_{(z_1,\ldots,z_k)\in\outext(S)}\bigg\{\sum_{j\in[k]:z_j\ge1} \ol{\Lambda}_{n_j}([S]_{j,z_j-1})\bigg\}+2t-2,
	\]
	that is, \(\msf{h}_S(\mb{X})\) attains the lower bound.
\end{proof}

\section{Partial results for \((t,0)\)-exact polynomial covers}\label{sec:partial-results}

Let us have a few notations before we proceed.  For any \(\alpha\in\mb{N}^n\), we denote \(\alpha!\coloneqq\alpha_1\cdots\alpha_n\).  Further, for any \(a\in\mb{R}^n\), we define the polynomial \((\mb{X}-a)^\alpha\coloneqq(X_1-a_1)^{\alpha_1}\cdots(X_n-a_n)^{\alpha_n}\).  Therefore, for any \(P(\mb{X})\in\mb{R}[\mb{X}]\) and \(a\in\mb{R}^n\), the Taylor's expansion of \(P\) about the point \(a\) is
\[
P(\mb{X})=\sum_{0\le|\alpha|\le\deg(P)}\frac{(\partial^\alpha P)(a)}{\alpha!}(\mb{X}-a)^\alpha.
\]

Let us first prove Proposition~\ref{pro:EPC-hamming}.  We mention the statement again, for convenience.
\begin{repproposition}[\ref{pro:EPC-hamming}]
	For \(w\in[1,n-1]\), let \(S\subsetneq\{0,1\}^n\) be the symmetric set defined by \(W_n(S)=[0,w-1]\).  Then for any \(t\in\big[2,\big\lfloor\frac{n+3}{2}\big\rfloor\big]\), we have
	\[
	\EPC_n^{(t,0)}(S)=w+2t-3.
	\]
	Further, the answer to Question~\ref{ques:main} is negative, in general.
\end{repproposition}
\begin{proof}
	Let \(P(\mb{X})\in\mb{R}[\mb{X}]\) be a \((t,0)\)-exact polynomial cover for \(S\).  Consider the restricted polynomial \(\wt{P}(X_1,\ldots,X_w)\coloneqq P(X_1,\ldots,X_w,0^{n-w})\).  Then \(\wt{P}(1^w)=P(1^w0^{n-w})\ne0\), since \(1^w0^{n-w}\not\in S\).  Further, for any \(x\in\{0,1\}^w,\,x\ne1^w\), we have \(|x|\le w-1\), which means \(x0^{n-w}\in S\), and so \(\mult(\wt{P}(X_1,\ldots,X_w),x)\ge\mult(P(\mb{X}),x0^{n-w})\ge t\).  Thus, \(\wt{P}(X_1,\ldots,X_w)\) is a \((t,0)\)-exact polynomial cover for \(\{0,1\}^w\setminus\{1^w\}\).  So by Theorem~\ref{thm:sauermann-wigderson-again}, we get \(\deg(\wt{P})\ge w+2t-3\).  Hence, \(\deg(P)\ge\deg(\wt{P})\ge w+2t-3\).
	
	In order to show that the lower bound is tight, let \(Q(X_1,\ldots,X_w)\in\mb{R}[\mb{X}]\) be a \((t,0)\)-exact polynomial cover of the symmetric set \(\{0,1\}^w\setminus\{1^w\}\), with \(\deg(Q)=w+2t-3\), as ensured by Theorem~\ref{thm:sauermann-wigderson-again}.  Further, let \(\alpha=Q(1^w)\ne0\).  Now define a polynomial
	\[
	\wt{Q}(\mb{X})=\sum_{1\le i_1<\cdots<i_w\le n}Q(X_{i_1},\ldots,X_{i_w}).
	\]
	Then clearly, \(\deg(\wt{Q})=\deg(Q)=w+2t-3\).  Now consider any \(x\in\{0,1\}^n\).  We observe the following.
	\begin{itemize}
		\item  If \(|x|\le w-1\), then for any \(1\le i_1<\cdots<i_w\le n\), we have
		\[
		\mult(\wt{Q}(\mb{X}),x)\ge\mult(Q(X_{i_1},\ldots,X_{i_w}),(x_{i_1},\ldots,x_{i_w}))\ge t.
		\]
		\item  If \(|x|\ge w\), let \(\{j_1,\ldots,j_u\}=\{i\in[n]:x_i=1\}\).  Then we have
		\[
		\wt{Q}(x)=\sum_{\substack{1\le i_1<\cdots<i_w\le n\\\{i_1,\ldots,i_w\}\subseteq\{j_1,\ldots,j_u\}}}\!\!\!\alpha\n=\binom{u}{w}\alpha\ne0.
		\]
	\end{itemize}
	Thus, \(\wt{Q}(\mb{X})\) is a \((t,0)\)-exact polynomial cover for \(S\).  Hence, \(\wt{Q}(\mb{X})\) attains the lower bound.
	
	Let us now show an example that illustrates that the answer to Question~\ref{ques:main} is negative in general, that is, \(\EHC_n^{(t,0)}(S)>\EPC_n^{(t,0)}(S)\) in general.  Let \(t=2,\,n=3>2t-3\).  Let \(S\subseteq\{0,1\}^3\) be the symmetric set defined by \(W_n(S)=\{0,1\}\).  So \(S=\{(0,0,0),(1,0,0),(0,1,0),(0,0,1)\}\).  We have \(\EPC_3^{(2,0)}(S)=3\) (since \(w=1\) in this case).  Now suppose there are three hyperplanes \(\{h_1,h_2,h_3\}\) that form a \((2,0)\)-exact hyperplane cover for \(S\).  Without loss of generality, let \(h_1(0,0,0)=h_2(0,0,0)=0\).  So \(h_1(X_1,X_2,X_3)=a_1X_1+a_2X_2+a_3X_3,\,h_2(X_1,X_2,X_3)=b_1X_1+b_2X_2+b_3X_3\), for some \(a_1,a_2,a_3,b_1,b_2,b_3\in\mb{R}\).  Further, without loss of generality, suppose \(h_1(1,0,0)=h_1(0,1,0)=0\).  This gives \(h_1(1,1,0)=h_1(1,0,0)+h_1(0,1,0)=0\), which contradicts \(\{h_1,h_2,h_3\}\) being a \((2,0)\)-exact hyperplane cover for \(S\).  So we conclude that \(|\mc{Z}(h_j)\cap\{(1,0,0),(0,1,0),(0,0,1)\}|\le 1\) for \(j\in[3]\).  This implies that it is impossible for each point in \(\{(1,0,0),(0,1,0),(0,0,1)\}\) to be covered 2 times by \(\{h_1,h_2,h_3\}\).  Hence \(\EHC_3^{(2,0)}(S)>3\).
\end{proof}

Now let us prove Proposition~\ref{pro:EPC-layer}.  We mention the statement again, for convenience.
\begin{repproposition}[\ref{pro:EPC-layer}]
	For any layer \(S\subsetneq\{0,1\}^n\) with \(W_n(S)=\{w\}\), and \(t\ge1\), we have
	\[
	\EPC_n^{(t,0)}(S)=t.
	\]
\end{repproposition}
\begin{proof}
	Let \(P(\mb{X})\in\mb{R}[\mb{X}]\) be a \((t,0)\)-exact polynomial cover for \(S\).  Fix any \(a\in S\).  So there exists an invertible linear map \(T:\mb{R}^n\to\mb{R}^n\) (or equivalently, an invertible change of coordinates) such that the polynomial \(\wt{P}(\mb{X})\coloneqq P(L(\mb{X}-a)+a)\) is a \((t,0)\)-exact polynomial cover for \(\{a\}\).  Also, \(\deg(\wt{P})=\deg(P)\).  Now we have the Taylor's expansion of \(\wt{P}\) about \(a\) as
	\[
	\wt{P}(\mb{X})=\sum_{0\le|\alpha|\le\deg(P)}\frac{(\partial^\alpha\wt{P})(a)}{\alpha!}(\mb{X}-a)^\alpha.
	\]
	Since \(\wt{P}(\mb{X})\) is a \((t,0)\)-exact polynomial cover for \(\{a\}\), we have \((\partial^\alpha\wt{P})(a)=0\) for \(|\alpha|<t\).  This gives
	\[
	\wt{P}(\mb{X})=\sum_{t\le|\alpha|\le\deg(P)}\frac{(\partial^\alpha\wt{P})(a)}{\alpha!}(\mb{X}-a)^\alpha,
	\]
	which implies \(\deg(P)\ge t\).  Further, this lower bound is tight; for instance, the polynomial \(P(\mb{X})\coloneqq\big(\sum_{i=1}^nX_i-w\big)^t\) witnesses the lower bound.
\end{proof}

\section{Conclusion and open questions}\label{sec:conclusion}

In this work, we have proved Theorem~\ref{thm:EPC-PDC}, which also subsumes our other main results (Theorem~\ref{thm:multiplicity-symmetric} and Theorem~\ref{thm:multiplicity-block-symmetric}).  We note that Theorem~\ref{thm:EPC-PDC} characterizes the tight bound for the \((t,t-1)\)-block exact polynomial cover, but in the special cases of Theorem~\ref{thm:multiplicity-symmetric} and Theorem~\ref{thm:multiplicity-block-symmetric}, our tight example specializes to the tight example for the \((t,t-1)\)-exact hyperplane cover.  Therefore, as seen in the earlier work of Alon and F\"uredi~\cite{alon-furedi} as well as initial attempts in~\cite{venkitesh-2022-covering,ghosh-kayal-nandi-2023-covering}, solving the \emph{weaker} polynomial covering problem by the polynomial method indeed solves the \emph{stronger} hyperplane covering problem in these settings.

Some of the obvious questions that seem beyond the proof technique employed in this work are the following.
\begin{question}
	In the broad generality of the polynomial covering problem considered for PDC blockwise symmetric sets, is the \emph{degeneracy condition} necessary?  More precisely, for any nonempty PDC \(k\)-wise symmetric set \(S\subseteq\{0,1\}^N\) and \(t\ge1\), is it true that \(\bEPC_{(n_1,\ldots,n_k)}^{(t,t-1)}(S)=\EPC_N^{(t,t-1)}(S)\)?
	
	We believe this could be true.
\end{question}
\begin{question}
	How do we solve the hyperplane covering problem considered for PDC blockwise symmetric sets?  In other words, do Theorem~\ref{thm:multiplicity-symmetric} and Theorem~\ref{thm:multiplicity-block-symmetric} extend to the setting of Theorem~\ref{thm:EPC-PDC} for the exact hyperplane cover version?  More precisely, for any nonempty PDC \(k\)-wise symmetric set \(S\subseteq\{0,1\}^N\) and \(t\ge1\), is it true that \(\bEPC_{(n_1,\ldots,n_k)}^{(t,t-1)}(S)=\bEHC_{(n_1,\ldots,n_k)}^{(t,t-1)}(S)\)?
	
	Our work shows that our proof technique can not possibly extend to prove this.  We therefore believe this may not be true.
\end{question}
\begin{question}
	Characterize the index complexity of all nonempty PDC \(k\)-wise symmetric sets.  In other words, obtain the characterization without requiring the \emph{outer intact} condition in Proposition~\ref{pro:PDC-index-complexity}.
\end{question}

\appendix

\section{Covering by hyperplanes and polynomials}

Let us give quick proofs of some simple statements that broadly assert that a hyperplane cover is a stronger notion than a polynomial cover.

\subsection{Exact hyperplane and polynomial covers: \(\EHC^{(t,\ell)}_n\ge\EPC^{(t,\ell)}_n\)}\label{app:EHC-EPC}

Consider \(S\subsetneq\{0,1\}^n\), and \(t\ge1,\,\ell\in[0,t-1]\).  Let us prove that \(\EHC_n^{(t,\ell)}(S)\ge\EPC_n^{(t,\ell)}(S)\).  Let \(\mc{H}(\mb{X})=\{H_1(\mb{X}),\ldots,H_k(\mb{X})\}\) be a \((t,\ell)\)-exact hyperplane cover for \(S\).  We have
\begin{itemize}
	\item  \(|\{i\in[k]:H_i(a)=0\}|\ge t\) for all \(a\in S\).
	\item  \(|\{i\in[k]:H_i(b)=0\}|=\ell\) for all \(b\in\{0,1\}^n\setminus S\).
\end{itemize}
Now consider the polynomial \(\mc{H}(X)=H_1(\mb{X})\cdots H_k(\mb{X})\).
\begin{enumerate}[(a)]
	\item  Fix any \(a\in S\) and \(\alpha\in\mb{N}^n,\,|\alpha|\le t-1\).  We have by the product rule for derivatives,
	\[
	(\partial^\alpha\mc{H})(a)=\sum_{\substack{\gamma^{(1)},\ldots,\gamma^{(k)}\in\mb{N}^n\\\gamma^{(1)}+\cdots+\gamma^{(k)}=\alpha}}\binom{\alpha}{\gamma^{(1)}\n\cdots\n\gamma^{(k)}}(\partial^{\gamma^{(1)}}H_1)(a)\cdots(\partial^{\gamma^{(k)}}H_k)(a).
	\]
	For each \(\gamma^{(1)},\ldots,\gamma^{(k)}\in\mb{N}^n\) with \(\gamma^{(1)}+\cdots+\gamma^{(k)}=\alpha\), since \(|\{i\in[k]:H_i(a)=0\}|\ge t\) and \(|\gamma^{(1)}|+\cdots+|\gamma^{(k)}|=|\alpha|\le t-1\), there exists \(i\in[k]\) such that \(\gamma^{(i)}=0^n\).  This implies \((\partial^{\gamma^{(1)}}H_1)(a)\cdots(\partial^{\gamma^{(k)}}H_k)(a)=0\).  Thus, \((\partial^\alpha\mc{H})(a)=0\).
	
	\item  Fix any \(b\in\{0,1\}^n\setminus S\).  Since \(|\{i\in[k]:H_i(b)=0\}|=\ell\), by the argument above, we get \((\partial^\beta\mc{H})(b)=0\) for every \(\beta\in\mb{N}^n,\,|\beta|\le\ell-1\).  Now recall that the collection \(\mc{H}(\mb{X})\) is a multiset, and suppose we alternatively represent \(\mc{H}(\mb{X})=\{(F_1(\mb{X}))^{(m_1)},\ldots,(F_v(\mb{X}))^{(m_v)}\}\), where \(F_1(\mb{X}),\ldots,F_v(\mb{X})\) are distinct, and \((F_u(\mb{X}))^{(m_u)}\) (for \(u\in[v]\)) indicates \(m_u\) many copies of \(F_u(\mb{X})\).  Then the condition \(|\{i\in[k]:H_i(b)=0\}|=\ell\) implies that there exists a subset \(U\subseteq[v]\) such that \(\sum_{u\in U}m_u=\ell\), and \(F_u(b)=0\) exactly when \(u\in U\).  Further, by definition, we also have the inequality \(\ell\le t-1<k\).  This means \(U\subsetneq[v]\), and \(F_{u'}(b)\ne0\) for all \(u'\in[v]\setminus U\).  Without loss of generality, we may assume \(U=[v']\) for some \(v'\in[0,v-1]\).
	
	So we have the following.
	\begin{itemize}
		\item  \(F_u(b)=0\) if \(u\in[v']\), and \(F_u(b)\ne0\) if \(u\in[v'+1,v]\).
		\item  \(\sum_{u=1}^{v'}m_u=\ell\).
	\end{itemize}
	
	Now for each \(u\in[v']\), since \(F_u(\mb{X})\) is an affine linear polynomial, let \(i_u\in[n]\) be the least such that \(\coeff(X_{i_u},F_u)\ne0\).  Define \(\nu=\sum_{u=1}^{v'}m_ue_{i_u}\).  Then we get
	\[
	(\partial^\nu\mc{H})(b)=\prod_{u=1}^{v'}(\coeff(X_{i_u},F_u))^{m_u}\ne0,
	\]
	where \(|\nu|=\sum_{u=1}^{v'}m_u=\ell\).
\end{enumerate}
Thus, \(\mc{H}(\mb{X})\) is a \((t,\ell)\)-exact polynomial cover for \(S\).  This completes the proof.

\subsection{Block exact hyperplane and polynomial covers: \(\bEHC^{(t,\ell)}_{(n_1,\ldots,n_k)}\ge\bEPC^{(t,\ell)}_{(n_1,\ldots,n_k)}\)}\label{app:bEHC-bEPC}

Let \(\{0,1\}^N=\{0,1\}^{n_1}\times\cdots\times\{0,1\}^{n_k}\).  Consider \(S\subsetneq\{0,1\}^N\), and \(t\ge1,\,\ell\in[0,t-1]\).  Let us prove that \(\bEHC_{(n_1,\ldots,n_k)}^{(t,\ell)}(S)\ge\bEPC_{(n_1,\ldots,n_k)}^{(t,\ell)}(S)\).  Let \(\mc{H}(\mb{X})=\{H_1(\mb{X}),\ldots,H_k(\mb{X})\}\) be a \((t,\ell)\)-block exact hyperplane cover for \(S\).  We have
\begin{itemize}
	\item  \(|\{i\in[k]:H_i(a)=0\}|\ge t\) for all \(a\in S\).
	\item  \(|\{i\in[k]:H_i(b)=0\}|=\ell\) for all \(b\in\{0,1\}^N\setminus S\).
	\item  for each \(j\in[k]\), and every \(a\in\{0,1\}^{n_1}\times\cdots\times\{0,1\}^{n_{j-1}}\times\{0,1\}^{n_{j+1}}\times\cdots\times\{0,1\}^{n_k}\), we have \(|\mc{H}(a,\mb{X}_j)|=|\mc{H}(\mb{X})|\).
\end{itemize}
Now consider the polynomial \(\mc{H}(\mb{X})=H_1(\mb{X})\cdots H_k(\mb{X})\).
\begin{enumerate}[(a)]
	\item  repeating the argument as in Appendix~\ref{app:EHC-EPC}(a), we can show that \((\partial^\alpha\mc{H})(a)=0\) for any \(a\in S\) and \(\alpha\in\mb{N}^N,\,|\alpha|\le t-1\).  So \(\mult(\mc{H}(\mb{X}),a)\ge t\) for all \(a\in S\).
	
	\item  Fix any \(b\in\{0,1\}^N\setminus S\).  Again, repeating the argument as in Appendix~\ref{app:EHC-EPC}(b), we can show that \((\partial^\beta\mc{H})(b)=0\) for every \(\beta\in\mb{N}^N,\,|\beta|\le\ell-1\).  Further, now fix any \(j\in[k]\), and denote \(b=(b',\wt{b})\), where \(b'\in\{0,1\}^{n_1}\times\cdots\times\{0,1\}^{n_{j-1}}\times\{0,1\}^{n_{j+1}}\times\cdots\times\{0,1\}^{n_k},\,\wt{b}\in\{0,1\}^{n_j}\).  This immediately gives \(\partial^{\wt{\beta}}\mc{H}(b',\mb{X}_j)\big|_{\wt{b}}=0\) for every \(\wt{\beta}\in\mb{N}^{n_j},\,|\wt{\beta}|\le\ell-1\).  Further, we have \(|\mc{H}(b',\mb{X}_j)|=|\mc{H}(\mb{X})|\).  This implies
	\[
	|\{H(b',\mb{X}_j)\in\mc{H}(b',\mb{X}_j):H(b',\wt{b})=0\}|=|\{H(\mb{X})\in\mc{H}(\mb{X}):H(b)=0\}|=\ell.
	\]
	Now repeating the argument as in Appendix~\ref{app:EHC-EPC}(b) over the hypercube \(\{0,1\}^{n_j}\), for the point \(\wt{b}\in\{0,1\}^{n_j}\), we can show that there exists \(\nu\in\mb{N}^{n_j},\,|\nu|=\ell\) such that \(\partial^{\nu}\mc{H}(b',\mb{X}_j)\big|_{\wt{b}}\ne0\).  So \(\mult(\mc{H}(b',\mb{X}_j),\wt{b})=\ell\).
\end{enumerate}
Thus, \(\mc{H}(\mb{X})\) is a \((t,\ell)\)-block exact polynomial cover for \(S\).  This completes the proof.

\section{An exact hyperplane cover: proof of Lemma~\ref{lem:T}}\label{app:GKN-hyperplane}

Let us give a proof of the exact hyperplane cover for the symmetric set \(T_{n,i}\subseteq\{0,1\}^n,\,i\in[0,\lceil n/2\rceil]\), constructed by~\cite{ghosh-kayal-nandi-2023-covering}, given in Lemma~\ref{lem:T}.  We mention the statement again, for convenience.
\begin{replemma}[\ref{lem:T}]
	For \(i\in[0,\lceil n/2\rceil]\), the collection of hyperplanes \(\{H^*_{(i,j)}(\mb{X}):j\in[i]\}\) defined by
	\[
	H^*_{(i,j)}(\mb{X})=\sum_{k=1}^{n-j}X_k-(n-2i+j)X_{n-j+1}-(i-j),\quad j\in[i],
	\]
	satisfies the following.
	\begin{itemize}
		\item[\(\bullet\)]  For every \(a\in T_{n,i}\), there exists \(j\in[i]\) such that \(H^*_{(i,j)}(a)=0\).
		\item[\(\bullet\)]  \(H^*_{(i,j)}(b)\ne0\) for every \(b\in\{0,1\}^n\setminus T_{n,i},\,j\in[i]\).
	\end{itemize}
\end{replemma}
\begin{proof}
	For any \(a\in\{0,1\}^n\), denote \(I_0(a)=\{t\in[n]:a_t=0\}\) and \(I_1(a)=\{t\in[n]:a_t=1\}\).  Consider any \(a\in T_{n,i}\).  So \(|a|\in[0,i-1]\cup[n-i+1,n]\).  We have two cases.
	\begin{enumerate}[(a)]
		\item  \(|a|\in[0,i-1]\).  Then \(|I_0(a)|\ge n-i+1\).  Let \(t_0\) be the \((n-i+1)\)-th element in \(I_0(a)\).  This means \(t_0\in[n-i+1,n]\), which implies that there exists \(j\in[i]\) such that \(t_0=n-j+1\).  So \(a_{n-j+1}=a_{t_0}=0\).  Further, by definition of \(t_0\), we get
		\[
		|a_{[1,n-j]}|=|a_{[1,n-j+1]}|=(n-j+1)-(n-i+1)=i-j.
		\]
		Thus, \(H^*_{(i,j)}(a)=(i-j)-(n-2i+j)\cdot0-(i-j)=0\).
		
		\item  \(|a|\in[n-i+1,n]\).  Then \(|I_1(a)|\ge n-i+1\).  Let \(t_1\) be the \((n-i+1)\)-th element in \(I_1(a)\).  This means \(t_1\in[n-i+1,n]\), which implies that there exists \(j\in[i]\) such that \(t_0=n-j+1\).  So \(a_{n-j+1}=a_{t_1}=1\).  Further, by definition of \(t_1\), we get
		\[
		|a_{[1,n-j]}|=|a_{[1,n-j+1]}|-1=(n-i+1)-1=n-i.
		\]
		Thus, \(H^*_{(i,j)}(a)=(n-i)-(n-2i+j)\cdot1-(i-j)=0\).  
	\end{enumerate}
	Now consider any \(b\in\{0,1\}^n\setminus T_{n,i}\).  So \(|b|\in[i,n-i]\).  Fix any \(j\in[i]\).  We have two cases.
	\begin{enumerate}[(a)]
		\item  \(b_{n-j+1}=0\).  Then \(|b_{[1,n-j]}|\in[i,n-i]\), and so
		\[
		H^*_{(i,j)}(b)\in[i-(i-j),n-i-(i-j)]=[j,n-2i+j],
		\]
		which implies \(H^*_{(i,j)}(b)\ge j\ge1\).
		
		\item  \(b_{n-j+1}=1\).  Then \(|b_{[1,n-j]}|\in[i-1,n-i-1]\), and so
		\[
		H^*_{(i,j)}(b)\in[i-1-(n-2i+j)-(i-j),n-i-1-(n-2i+j)-(i-j)]=[2i-n-1,-1],
		\]
		which implies \(H^*_{(i,j)}(b)\le-1\).
	\end{enumerate}
	This completes the proof.
\end{proof}

\section{Inner and outer intervals of symmetric sets}\label{app:inner-outer}

Here we discuss some ancillary details about inner and outer intervals of symmetric sets.

\subsection{Uniqueness of inner and outer intervals}

Let \(S\subseteq\{0,1\}^n\) be a symmetric set.  It is immediate that \(\innint(\emptyset)=\outint(\emptyset)=J_{n,-1,n+1}=\emptyset\).  It is also immediate that \(\innint(\{0,1\}^n)=\outint(\{0,1\}^n)=J_{n,\lfloor n/2\rfloor,\lfloor n/2\rfloor+1}=\{0,1\}^n\).

Now consider a nonempty symmetric set \(S\subsetneq\{0,1\}^n\).  We note the following.
\begin{itemize}
	\item  There exists \(w\in[0,n]\) such that \(w\not\in W_n(S)\), that is, \(W_n(S)\subseteq[0,w-1]\cup[w+1,n]\).  Let
	\begin{align*}
		a&=\max\{a'\in[0,w-1]:[0,a]\subseteq W_n(S)\}&\tx{with the convention }\max(\emptyset)&=-1,\\
		b&=\min\{b'\in[w+1,n]:[b',n]\subseteq W_n(S)\}&\tx{with the convention }\min(\emptyset)&=n+1.
	\end{align*}
	Then it follows immediately by definition that \(\innint(S)=J_{n,a,b}\), and is therefore unique.
	
	\item  Recall that \(\mc{O}(S)\) is the collection of all peripheral intervals \(J_{n,a,b}\) such that \(S\subseteq J_{n,a,b}\) and \(I_{n,a,b}=W_n(J_{n,a,b})\) has minimum size.  Now consider the function \(\lambda_S:\mc{O}(S)\to\mb{N}\) defined by
	\[
	\lambda_S(J_{n,a,b})=|a+b-n|,\quad\tx{for all }J_{n,a,b}\in\mc{O}(S).
	\]
	We observe a simple property of the minimizer of \(\lambda_S\).
	\begin{observation}\label{obs:outer-unique}
		The minimizer of \(\lambda_S\) is either a unique peripheral interval \(J_{n,a,b}\), or exactly a pair of peripheral intervals \(\{J_{n,a,b},J_{n,n-b,n-a}\}\).  
	\end{observation}
	\begin{proof}
		Suppose the minimizer of \(\lambda_S\) is not unique, that is, there are two distinct minimizers \(J_{n,a,b},\,J_{n,a',b'}\in\mc{O}(S)\).  Then by definition of \(\mc{O}(S)\), we already have \(|I_{n,a,b}|=|I_{n,a',b'}|\), which implies \(a-b=a'-b'\).  So there exists \(h\in\mb{Z}\) such that \(a'=a+h,\,b'=b+h\).  Further, by the minimization of \(\lambda_S\), we have \(|a+b-n|=|a'+b'-n|\), which yields two cases.
		\begin{enumerate}[(a)]
			\item  \(a+b-n=a'+b'-n\), that is, \(a+b=a'+b'\).  This implies \(h=0\), and so \(a'=a,\,b'=b\).
			\item  \(a+b-n=n-a'-b'\), that is, \(a+b=2n-(a'+b')\).  This implies \(h=n-(a+b)\), and so \(a'=n-b,\,b'=n-a\).
		\end{enumerate}
		This completes the proof.
	\end{proof}
	Recall that \(\outint(S)\) is defined by
	\[
	\outint(S)=\begin{cases}
		J_{n,a,b}&\tx{if }J_{n,a,b}\tx{ is the unique minimizer of }\lambda_S,\\
		J_{n,a,b}&\tx{if }\{J_{n,a,b},J_{n,n-b,n-a}\}\tx{ are minimizers of }\lambda_S\tx{ and }a>n-b.
	\end{cases}
	\]
	Thus, by Observation~\ref{obs:outer-unique}, it is immediate that \(\outint(S)\) is unique.
\end{itemize}

\subsection{Illustrations of inner and outer intervals}

Let us illustrate some examples of inner and outer intervals.  Figure~\ref{fig:inner-outer-interval} shows two typical symmetric sets -- \emph{one-sided} and \emph{two-sided} -- and their inner and outer intervals.

	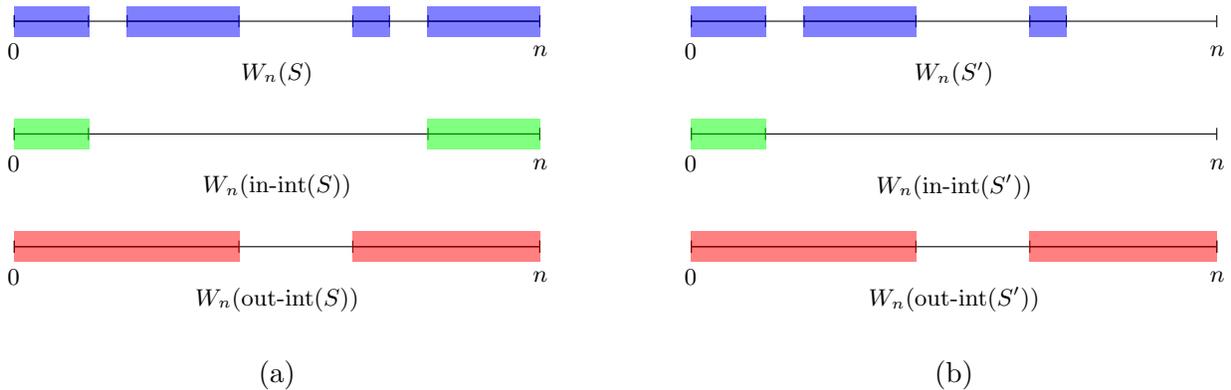
\begin{figure}[htbp]
		\begin{tikzpicture}
			\draw[|-|] (0,0)--(7,0);
			\draw[|-|] (0,0)--(1,0);
			\fill[opacity=0.5,blue] (0,0.2)--(1,0.2)--(1,-0.2)--(0,-0.2);
			\draw[|-|] (1.5,0)--(3,0);
			\fill[opacity=0.5,blue] (1.5,0.2)--(3,0.2)--(3,-0.2)--(1.5,-0.2);
			\draw[|-|] (4.5,0)--(5,0);
			\fill[opacity=0.5,blue] (4.5,0.2)--(5,0.2)--(5,-0.2)--(4.5,-0.2);
			\draw[|-|] (5.5,0)--(7,0);
			\fill[opacity=0.5,blue] (5.5,0.2)--(7,0.2)--(7,-0.2)--(5.5,-0.2);
			\node at (0,-0.4) {\footnotesize\(0\)};
			\node at (7,-0.4) {\footnotesize\(n\)};
			\node at (3.5,-0.5-0.2) {\footnotesize\(W_n(S)\)};
			
			\draw[|-|] (0,-1.5)--(7,-1.5);
			\draw[|-|] (0,0-1.5)--(1,0-1.5);
			\fill[opacity=0.5,green] (0,0.2-1.5)--(1,0.2-1.5)--(1,-0.2-1.5)--(0,-0.2-1.5);
			\draw[|-|] (5.5,0-1.5)--(7,0-1.5);
			\fill[opacity=0.5,green] (5.5,0.2-1.5)--(7,0.2-1.5)--(7,-0.2-1.5)--(5.5,-0.2-1.5);
			\node at (0,-0.4-1.5) {\footnotesize\(0\)};
			\node at (7,-0.4-1.5) {\footnotesize\(n\)};
			\node at (3.5,-0.5-1.5-0.2) {\footnotesize\(W_n(\innint(S))\)};
			
			\draw[|-|] (0,0-3)--(7,0-3);
			\draw[|-|] (0,0-3)--(3,0-3);
			\fill[opacity=0.5,red] (0,0.2-3)--(3,0.2-3)--(3,-0.2-3)--(0,-0.2-3);
			\draw[|-|] (4.5,0-3)--(7,0-3);
			\fill[opacity=0.5,red] (4.5,0.2-3)--(7,0.2-3)--(7,-0.2-3)--(4.5,-0.2-3);
			\node at (0,-0.4-3) {\footnotesize\(0\)};
			\node at (7,-0.4-3) {\footnotesize\(n\)};
			\node at (3.5,-0.5-3-0.2) {\footnotesize\(W_n(\outint(S))\)};
			\node at (3.5,-0.5-3-0.2-1) {(a)};
			
			\draw[|-|] (0+9,0)--(7+9,0);
			\draw[|-|] (0+9,0)--(1+9,0);
			\fill[opacity=0.5,blue] (0+9,0.2)--(1+9,0.2)--(1+9,-0.2)--(0+9,-0.2);
			\draw[|-|] (1.5+9,0)--(3+9,0);
			\fill[opacity=0.5,blue] (1.5+9,0.2)--(3+9,0.2)--(3+9,-0.2)--(1.5+9,-0.2);
			\draw[|-|] (4.5+9,0)--(5+9,0);
			\fill[opacity=0.5,blue] (4.5+9,0.2)--(5+9,0.2)--(5+9,-0.2)--(4.5+9,-0.2);
			\node at (0+9,-0.4) {\footnotesize\(0\)};
			\node at (7+9,-0.4) {\footnotesize\(n\)};
			\node at (3.5+9,-0.5-0.2) {\footnotesize\(W_n(S')\)};
			
			\draw[|-|] (0+9,-1.5)--(7+9,-1.5);
			\draw[|-|] (0+9,0-1.5)--(1+9,0-1.5);
			\fill[opacity=0.5,green] (0+9,0.2-1.5)--(1+9,0.2-1.5)--(1+9,-0.2-1.5)--(0+9,-0.2-1.5);
			\node at (0+9,-0.4-1.5) {\footnotesize\(0\)};
			\node at (7+9,-0.4-1.5) {\footnotesize\(n\)};
			\node at (3.5+9,-0.5-1.5-0.2) {\footnotesize\(W_n(\innint(S'))\)};
			
			\draw[|-|] (0+9,0-3)--(7+9,0-3);
			\draw[|-|] (0+9,0-3)--(3+9,0-3);
			\fill[opacity=0.5,red] (0+9,0.2-3)--(3+9,0.2-3)--(3+9,-0.2-3)--(0+9,-0.2-3);
			\draw[|-|] (4.5+9,0-3)--(7+9,0-3);
			\fill[opacity=0.5,red] (4.5+9,0.2-3)--(7+9,0.2-3)--(7+9,-0.2-3)--(4.5+9,-0.2-3);
			\node at (0+9,-0.4-3) {\footnotesize\(0\)};
			\node at (7+9,-0.4-3) {\footnotesize\(n\)};
			\node at (3.5+9,-0.5-3-0.2) {\footnotesize\(W_n(\outint(S'))\)};
			\node at (3.5+9,-0.5-3-0.2-1) {(b)};
		\end{tikzpicture}
		\caption{(a) a \emph{two-sided} symmetric set \(S\), and (b) a \emph{one-sided} symmetric set \(S'\)}
		\label{fig:inner-outer-interval}
	\end{figure}
	
	\newpage
	Note that Figure~\ref{fig:inner-outer-interval} is a typical illustration.  The inner and outer intervals are special when \(S\) itself is either a peripheral interval or the complement of a peripheral interval.  Figure~\ref{fig:peripheral-interval-type-I} shows the inner and outer intervals of a \emph{two-sided} peripheral interval and its complement.

		\begin{figure}[htbp]
		\begin{tikzpicture}
			\draw[|-|] (0,0)--(7,0);
			\draw[|-|] (0,0)--(1,0);
			\fill[opacity=0.5,blue] (0,0.2)--(1,0.2)--(1,-0.2)--(0,-0.2);
			\draw[|-|] (4.5,0)--(7,0);
			\fill[opacity=0.5,blue] (4.5,0.2)--(7,0.2)--(7,-0.2)--(4.5,-0.2);
			\node at (0,-0.4) {\footnotesize\(0\)};
			\node at (7,-0.4) {\footnotesize\(n\)};
			\node at (3.5,-0.5-0.2) {\footnotesize\(W_n(J_{n,a,b})\)};
			
			\draw[|-|] (0,-1.5)--(7,-1.5);
			\draw[|-|] (0,0-1.5)--(1,0-1.5);
			\fill[opacity=0.5,green] (0,0.2-1.5)--(1,0.2-1.5)--(1,-0.2-1.5)--(0,-0.2-1.5);
			\draw[|-|] (4.5,0-1.5)--(7,0-1.5);
			\fill[opacity=0.5,green] (4.5,0.2-1.5)--(7,0.2-1.5)--(7,-0.2-1.5)--(4.5,-0.2-1.5);
			\node at (0,-0.4-1.5) {\footnotesize\(0\)};
			\node at (7,-0.4-1.5) {\footnotesize\(n\)};
			\node at (3.5,-0.5-1.5-0.2) {\footnotesize\(W_n(\innint(J_{n,a,b}))\)};
			
			\draw[|-|] (0,0-3)--(7,0-3);
			\draw[|-|] (0,0-3)--(1,0-3);
			\fill[opacity=0.5,red] (0,0.2-3)--(1,0.2-3)--(1,-0.2-3)--(0,-0.2-3);
			\draw[|-|] (4.5,0-3)--(7,0-3);
			\fill[opacity=0.5,red] (4.5,0.2-3)--(7,0.2-3)--(7,-0.2-3)--(4.5,-0.2-3);
			\node at (0,-0.4-3) {\footnotesize\(0\)};
			\node at (7,-0.4-3) {\footnotesize\(n\)};
			\node at (3.5,-0.5-3-0.2) {\footnotesize\(W_n(\outint(J_{n,a,b}))\)};
			\node at (3.5,-0.5-3-0.2-1) {(a)};
			
			\draw[|-|] (0+9,0)--(7+9,0);
			\draw[|-|] (1+9,0)--(4.5+9,0);
			\fill[opacity=0.5,blue] (1+9,0.2)--(4.5+9,0.2)--(4.5+9,-0.2)--(1+9,-0.2);
			\node at (0+9,-0.4) {\footnotesize\(0\)};
			\node at (7+9,-0.4) {\footnotesize\(n\)};
			\node at (3.5+9,-0.5-0.2) {\footnotesize\(W_n(\{0,1\}^n\setminus J_{n,a,b})\)};
			
			\draw[|-|] (0+9,-1.5)--(7+9,-1.5);
			\node at (0+9,-0.4-1.5) {\footnotesize\(0\)};
			\node at (7+9,-0.4-1.5) {\footnotesize\(n\)};
			\node at (3.5+9,-0.5-1.5-0.2) {\footnotesize\(W_n(\innint(\{0,1\}^n\setminus J_{n,a,b}))=\emptyset\)};
			
			\draw[|-|] (0+9,0-3)--(7+9,0-3);
			\draw[dotted,thick] (3.5+9,-0.2-3)--(3.5+9,0.2-3);
			\fill[opacity=0.5,red] (0+9,0.2-3)--(3.5+9,0.2-3)--(3.5+9,-0.2-3)--(0+9,-0.2-3);
			\fill[opacity=0.5,red] (3.5+9,0.2-3)--(7+9,0.2-3)--(7+9,-0.2-3)--(3.5+9,-0.2-3);
			\node at (0+9,-0.4-3) {\footnotesize\(0\)};
			\node at (7+9,-0.4-3) {\footnotesize\(n\)};
			\node at (3.5+9,-0.5-3-0.2) {\footnotesize\(W_n(\outint(\{0,1\}^n\setminus J_{n,a,b}))=[0,n]\)};
			\node at (3.5+9,-0.5-3-0.2-1) {(b)};
		\end{tikzpicture}
		\caption{(a) a \emph{two-sided} peripheral interval \(J_{n,a,b}\), and (b) the complement of \(J_{n,a,b}\)}
		\label{fig:peripheral-interval-type-I}
	\end{figure}
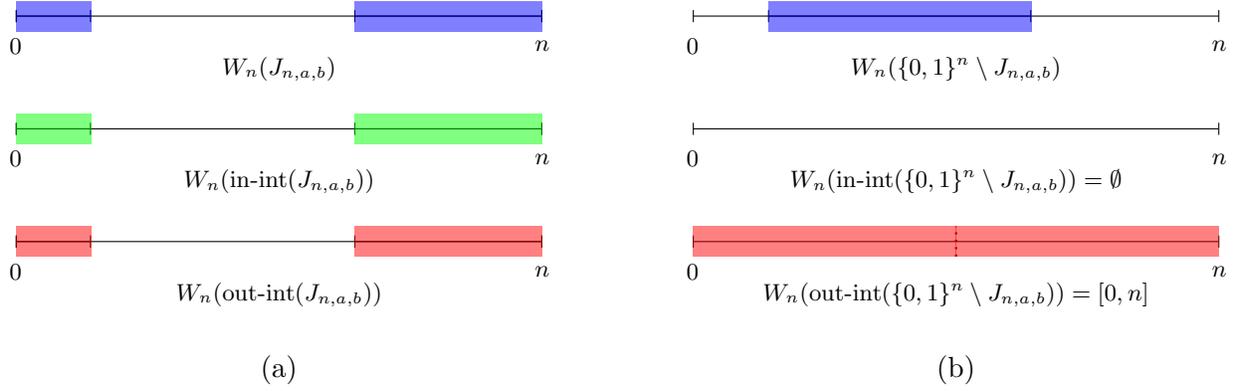

Figure~\ref{fig:peripheral-interval-type-II} shows the inner and outer intervals of a \emph{one-sided} peripheral interval and its complement.  Note that the complement of a \emph{one-sided} peripheral interval is again a \emph{one-sided} peripheral interval.

		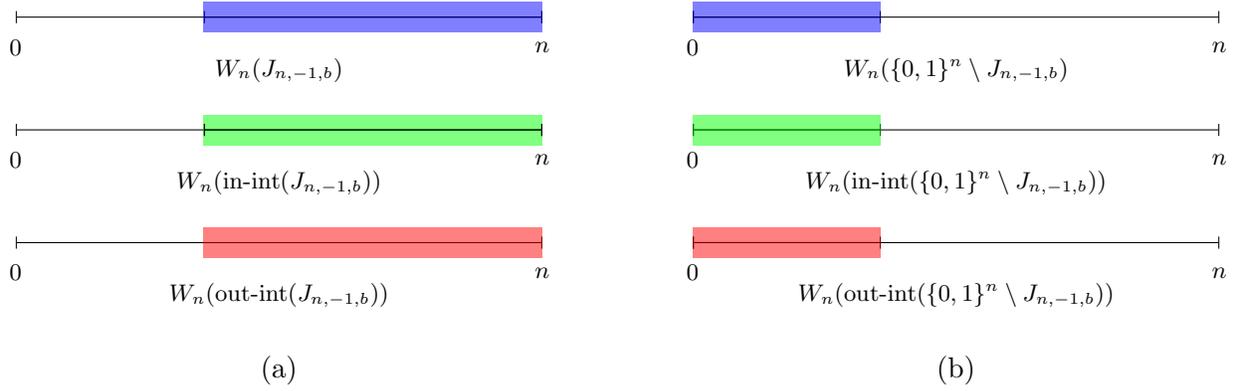
\begin{figure}[htbp]
		\begin{tikzpicture}
			\draw[|-|] (0,0)--(7,0);
			\draw[|-|] (2.5,0)--(7,0);
			\fill[opacity=0.5,blue] (2.5,0.2)--(7,0.2)--(7,-0.2)--(2.5,-0.2);
			\node at (0,-0.4) {\footnotesize\(0\)};
			\node at (7,-0.4) {\footnotesize\(n\)};
			\node at (3.5,-0.5-0.2) {\footnotesize\(W_n(J_{n,-1,b})\)};
			
			\draw[|-|] (0,-1.5)--(7,-1.5);
			\draw[|-|] (2.5,0-1.5)--(7,0-1.5);
			\fill[opacity=0.5,green] (2.5,0.2-1.5)--(7,0.2-1.5)--(7,-0.2-1.5)--(2.5,-0.2-1.5);
			\node at (0,-0.4-1.5) {\footnotesize\(0\)};
			\node at (7,-0.4-1.5) {\footnotesize\(n\)};
			\node at (3.5,-0.5-1.5-0.2) {\footnotesize\(W_n(\innint(J_{n,-1,b}))\)};
			
			\draw[|-|] (0,0-3)--(7,0-3);
			\draw[|-|] (2.5,0-1.5)--(7,0-1.5);
			\fill[opacity=0.5,red] (2.5,0.2-3)--(7,0.2-3)--(7,-0.2-3)--(2.5,-0.2-3);
			\node at (0,-0.4-3) {\footnotesize\(0\)};
			\node at (7,-0.4-3) {\footnotesize\(n\)};
			\node at (3.5,-0.5-3-0.2) {\footnotesize\(W_n(\outint(J_{n,-1,b}))\)};
			\node at (3.5,-0.5-3-0.2-1) {(a)};
			
			\draw[|-|] (0+9,0)--(7+9,0);
			\draw[|-|] (0+9,0)--(2.5+9,0);
			\fill[opacity=0.5,blue] (0+9,0.2)--(2.5+9,0.2)--(2.5+9,-0.2)--(0+9,-0.2);
			\node at (0+9,-0.4) {\footnotesize\(0\)};
			\node at (7+9,-0.4) {\footnotesize\(n\)};
			\node at (3.5+9,-0.5-0.2) {\footnotesize\(W_n(\{0,1\}^n\setminus J_{n,-1,b})\)};
			
			\draw[|-|] (0+9,-1.5)--(7+9,-1.5);
			\draw[|-|] (0+9,0-1.5)--(2.5+9,0-1.5);
			\fill[opacity=0.5,green] (0+9,0.2-1.5)--(2.5+9,0.2-1.5)--(2.5+9,-0.2-1.5)--(0+9,-0.2-1.5);
			\node at (0+9,-0.4-1.5) {\footnotesize\(0\)};
			\node at (7+9,-0.4-1.5) {\footnotesize\(n\)};
			\node at (3.5+9,-0.5-1.5-0.2) {\footnotesize\(W_n(\innint(\{0,1\}^n\setminus J_{n,-1,b}))\)};
			
			\draw[|-|] (0+9,0-3)--(7+9,0-3);
			\draw[|-|] (0+9,0-3)--(2.5+9,0-3);
			\fill[opacity=0.5,red] (0+9,0.2-3)--(2.5+9,0.2-3)--(2.5+9,-0.2-3)--(0+9,-0.2-3);
			\node at (0+9,-0.4-3) {\footnotesize\(0\)};
			\node at (7+9,-0.4-3) {\footnotesize\(n\)};
			\node at (3.5+9,-0.5-3-0.2) {\footnotesize\(W_n(\outint(\{0,1\}^n\setminus J_{n,-1,b}))\)};
			\node at (3.5+9,-0.5-3-0.2-1) {(b)};
		\end{tikzpicture}
		\caption{(a) a \emph{one-sided} peripheral interval \(J_{n,-1,b}\), and (b) the complement of \(J_{n,-1,b}\).  Note that the complement is \(\{0,1\}^n\setminus J_{n,-1,b}=J_{n,b-1,n+1}\)}
		\label{fig:peripheral-interval-type-II}
	\end{figure}

\newpage
\section{Invariance of \(\Lambda_n\) and \(r_n\) under complementation}\label{app:complement}

Let us quickly prove Fact~\ref{fact:transform}.  We can prove this fact by carefully following the definitions.  Instead, let us prove by using inner and outer intervals.  Let \(S\subseteq\{0,1\}^n\) be a symmetric set, and \(\wt{S}\) be the image of \(S\) under the coordinate transformation \((X_1,\ldots,X_n)\mapsto(1-X_1,\ldots,1-X_n)\).  This implies
\[
W_n(\wt{S})=\{n-w:w\in W_n(S)\}.
\]
So we get the following observations.
\begin{itemize}
	\item  If \(\outint(S)=J_{n,a,b}\), then \(\outint(\wt{S})=J_{n,n-b,n-a}\).
	\item  If \(\innint(S)=J_{n,a,b}\), then \(\innint(\wt{S})=J_{n,n-b,n-a}\).
\end{itemize}
Then, using Proposition~\ref{pro:index-complexity-symmetric} and Fact~\ref{fact:lambda} completes the proof of Fact~\ref{fact:transform}.

\raggedright
{\small\bibliographystyle{alpha}
\bibliography{references}}

\end{document}